\documentclass[12pt,twoside]{amsart}
\usepackage{amsmath,amsthm,amsmath,amsrefs}

\usepackage{calc}
\usepackage{setspace}
\usepackage{amsbsy,amsmath,amsfonts,amsthm,amssymb,color}

\theoremstyle{plain}
\newtheorem{mythe}{Theorem}[section]
\newtheorem{lem}[mythe]{Lemma}
\newtheorem{mydef}[mythe]{Definition}
\newtheorem{cor}[mythe]{Corollary}
\newtheorem{rem}[mythe]{Remark}
\newtheorem{pro}[mythe]{Proposition}

\theoremstyle{definition}

\usepackage{mathrsfs}
\newcommand{\ee}{\varepsilon}
\newcommand{\ran}{\text{ran}}
\newcommand{\bN}{\mathbb{N}}
\newcommand{\cT}{\mathcal{T}}

\newcommand{\cG}{\mathcal{G}}

\newcommand{\cS}{\mathcal{S}}
\newcommand{\cB}{\mathcal{B}}

\newcommand{\cK}{\mathcal{K}}

\newcommand{\la}{\langle}
\newcommand{\ra}{\rangle}
\newcommand{\fV}{\mathfrak{V}}
\newcommand{\cD}{\mathcal{D}}
\newcommand{\cX}{\mathcal{X}}
\newcommand{\cY}{\mathcal{Y}}
\newcommand{\cH}{\mathcal{H}}
\newcommand{\cA}{\mathcal{A}}
\newcommand{\er}{\text{er}}
\newcommand{\el}{\text{el}}
\newcommand{\bofh}{\cB(\cH)}

\newcommand{\cM}{\mathcal{M}}
\newcommand{\cU}{\mathcal{U}}
\newcommand{\bC}{\mathbb{C}}
\newcommand{\bR}{\mathbb{R}}
\newcommand{\id}{\text{id}}
\newcommand{\bZ}{\mathbb{Z}}
\newcommand{\cI}{\mathcal{I}}
\newcommand{\cJ}{\mathcal{J}}
\newcommand{\cV}{\mathcal{V}}

\newcommand{\tr}{\text{tr}}
\newcommand{\spn}{\text{span }}

\usepackage{tikz}

\newdimen\Squaresize \Squaresize=14pt
\newdimen\Thickness \Thickness=0.4pt

\def\Square#1{\hbox{\vrule width \Thickness
   \vbox to \Squaresize{\hrule height \Thickness\vss
      \hbox to \Squaresize{\hss#1\hss}
   \vss\hrule height\Thickness}
\unskip\vrule width \Thickness} \kern-\Thickness}

\def\Vsquare#1{\vbox{\Square{$#1$}}\kern-\Thickness}

\title[A Non-Commutative Unitary Analogue of Kirchberg's conjecture]{A Non-Commutative Unitary Analogue of Kirchberg's conjecture}

\author{Samuel J. Harris}
\thanks{The author was supported in part by NSERC (Canada).}
\address{University of Waterloo \\
Department of Pure Mathematics \\
200 University Ave. W. \\
Waterloo, Ontario \\
Canada  N2L 3G1\\}
\email{sj2harri@uwaterloo.ca}

\begin{document}

\begin{abstract}  The $C^{\ast}$-algebra $\cU_{nc}(n)$ is the universal $C^{\ast}$-algebra generated by $n^2$ generators $u_{ij}$ that make up a unitary matrix.  We prove that Kirchberg's formulation of Connes' embedding problem has a positive answer if and only if $\cU_{nc}(2) \otimes_{\min} \cU_{nc}(2)=\cU_{nc}(2) \otimes_{\max} \cU_{nc}(2)$.  Our results follow from properties of the finite-dimensional operator system $\cV_n$ spanned by $1$ and the generators of $\cU_{nc}(n)$.  We show that $\cV_n$ is an operator system quotient of $M_{2n}$ and has the OSLLP.  We obtain necessary and sufficient conditions on $\cV_n$ for there to be a positive answer to Kirchberg's problem.  Finally, in analogy with recent results of Ozawa, we show that a form of Tsirelson's problem related to $\cV_n$ is equivalent to Connes' embedding problem.
\end{abstract}


\maketitle

\tableofcontents

\section*{Introduction}

One of the most significant open problems in the field of operator algebras is Connes' embedding problem \cite{connes}, which asks whether every finite von Neumann algebra with separable predual can be embedded into the ultrapower of the hyperfinite $II_1$ factor in a trace-preserving way.  One of the simpler formulations of the problem is known as Kirchberg's problem \cite[Proposition 8]{kirchberg93}, which asks whether or not $C^*(F_n) \otimes_{\min} C^*(F_n)=C^*(F_n) \otimes_{\max} C^*(F_n)$ for some (equivalently all) $n \geq 2$, where $F_n$ is the free group on $n$ generators.  Due to recent work in \cite{fritz}, \cite{junge}, and \cite{ozawa}, another equivalent statement of the embedding problem is in terms of certain sets of quantum bipartite correlations (see \cite{fritz}, \cite{junge}, \cite{ozawa} and \cite{tsirelson} for more information on these correlations).

One of our main results is that Kirchberg's problem, stated in terms of $C^*(F_n)$, is equivalent to the same problem when $C^*(F_n)$ is replaced by Brown's $C^{\ast}$-algebra $\cU_{nc}(n)$, defined in \cite{brown} as the universal $C^{\ast}$-algebra whose generators make up a unitary $n \times n$ matrix.  In other words, Connes' embedding problem has a positive answer if and only if $\cU_{nc}(2) \otimes_{\min} \cU_{nc}(2)=\cU_{nc}(2) \otimes_{\max} \cU_{nc}(2)$ (see Theorem \ref{uncequivalenttokirchberg}).  On the way to proving this equivalence, we obtain a new proof of Kirchberg's theorem \cite[Corollary 1.2]{kirchberg} that $C^*(F_n) \otimes_{\min} \bofh=C^*(F_n) \otimes \bofh$ for every Hilbert space $\cH$.  Both of these results follow from properties of the finite-dimensional operator system $\cV_n$ spanned by the generators of $\cU_{nc}(n)$.  The significance of stating Kirchberg's conjecture in terms of $\cU_{nc}(n)$ lies in recent advances in quantum information theory.  Indeed, the phenomenon of embezzling entanglement in a bipartite scenario can be modelled using states on tensor products of $\cU_{nc}(n)$ (see \cite{CLP} for more information on embezzlement of entanglement).  Our study of the $C^{\ast}$-algebra $\cU_{nc}(n)$ and states on the tensor products $\cU_{nc}(n) \otimes_{\min} \cU_{nc}(n)$ and $\cU_{nc}(n) \otimes_{\max} \cU_{nc}(n)$ allow for a theory of so-called ``unitary correlation sets", which is motivated by their connection to quantum information theory.  A problem involving unitary correlation sets that is analogous to Tsirelson's problem is shown to also be equivalent to Kirchberg's problem.

The methods in this paper draw on many results in the recent theory of operator system tensor products and operator system quotients.  In Section $\S 1$ we review some basic results in the theory of operator system tensor products, and in Section $\S 2$ we recall some nuclearity-related properties of operator systems that arise in equivalent formulations of Kirchberg's problem.  Section $\S 3$ gives a short introduction to the theory of operator system quotients.  In Section $\S 4$, we explore properties of the operator system $\cV_n$ and the $C^{\ast}$-algebra $\cU_{nc}(n)$ while giving alternate characterizations of the WEP and DCEP in terms of tensor products with $\cV_n$.  We link both $\cV_n$ and $\cU_{nc}(n)$ to Kirchberg's problem in Section $\S 5$.  Finally, Section $\S 6$ draws on some results in quantum bipartite correlations in the field of quantum information theory, and an analogous theory of such correlations is developed in terms of the $C^{\ast}$-algebra $\cU_{nc}(n)$.

\section{Operator Systems and their Tensor Products}

In this section, we include a short introduction to operator systems and their tensor theory.  The interested reader can see \cite{KPTT} for a thorough introduction to the subject.  First, we give the abstract definition of an operator system, bearing in mind that concrete operator systems are always self-adjoint vector subspaces of $\bofh$, for some Hilbert space $\cH$, which contain the identity element.

Assume that $\cS$ is a complex vector space.  An \textbf{involution} on $\cS$ is a map $*:\cS \to \cS$ such that for all $x,y \in \cS$ and $\alpha \in \bC$,
\begin{itemize}
\item
$(\alpha x+y)^*=\overline{\alpha} x^*+y^*$, and
\item
$(x^*)^*=x$.
\end{itemize}
We denote by $\cS_h$ the set of all $x \in \cS$ with $x=x^*$.  We call $\cS$ a \textbf{$*$-vector space} if it is a complex vector space equipped with an involution.  Whenever $\cS$ is a $*$-vector space, there is a natural way to make $M_n(\cS)$ into a $*$-vector space, where $M_n(\cS)$ is the space of all $n \times n$ matrices with entries in $\cS$; indeed, one may let $(x_{ij})^*=(x_{ji}^*)$ for each $(x_{ij}) \in M_n(\cS)$.

Given a $*$-vector space $\cS$, a \textbf{matrix ordering} on $\cS$ is a set of cones $\{C_n\}_{n=1}^{\infty}$, where $C_n \subseteq (M_n(\cS))_h$, satisfying the following conditions:
\begin{itemize}
\item
$C_n+C_n \subseteq C_n$ and $tC_n \subseteq C_n$ for all $t \geq 0$ and $n \in \bN$,
\item
$A^*C_nA \subseteq C_m$ for all $n,m \in \bN$ and $A \in M_{n,m}(\bC)$, and
\item
$C_n \cap (-C_n)=\{0\}$ for all $n \in \bN$.
\end{itemize}
A $*$-vector space $\cS$ is said to be a \textbf{matrix-ordered $*$-vector space} if there is a matrix order $\{C_n\}_{n=1}^{\infty}$ on $\cS$.

We say that an element $e \in \cS_h$ is an \textbf{order unit} if for every $x \in \cS_h$, there is $r>0$ such that $x+re \in C_1$.  We call $e$ an \textbf{Archimedean order unit} if whenever $x \in \cS_h$ is such that $x+re \in C_1$ for all $r>0$, we have $x \in C_1$.  We call $e$ a \textbf{matrix order unit} if for each $n \in \bN$, $I_n=\begin{pmatrix} e \\ & \ddots \\ & & e \end{pmatrix} \in M_n(\cS)$ is an order unit for $M_n(\cS)$.  Note that $e$ is a matrix order unit if and only if it is an order unit for $\cS$ with respect to the cone $C_1$. We say that $e$ is an \textbf{Archimedean matrix order unit} provided that each $I_n$ is an Archimedean order unit.  With this terminology in hand, we can define the abstract version of operator systems.  An \textbf{(abstract) operator system} is a triple $(\cS,\{C_n\}_{n=1}^{\infty},e)$ where $\cS$ is a $*$-vector space, $\{C_n\}_{n=1}^{\infty}$ is a matrix ordering on $\cS$, and $e$ is an Archimedean matrix order unit for $\cS$.  We will usually refer to $\cS$ as an operator system, allowing the context to dictate which matrix ordering is being used.  If $(\cS,\{C_n\}_{n=1}^{\infty},e)$ and $(\cT,\{D_n\}_{n=1}^{\infty},e)$ are operator systems with $\cT \subseteq \cS$, we will say that $\cT$ is an \textbf{operator subsystem} of $\cS$ provided that $\cT$ has the same $*$-vector space structure as $\cS$ and $D_n=C_n \cap \cT$ for all $n \in \bN$.

Given operator systems $\cS$ and $\cT$ and a number $k \in \bN$, we say that a linear map $\varphi:\cS \to \cT$ is \textbf{$k$-positive} provided that for all $n \leq k$, $(\varphi(x_{ij})) \in M_n(\cT)_+$ whenever $(x_{ij}) \in M_n(\cS)_+$.  We say that $\varphi$ is \textbf{completely positive} if it is $k$-positive for every $k \in \bN$.  We use the abbreviaton ``ucp" for ``unital and completely positive".  A ucp map $\varphi:\cS \to \cT$ is a \textbf{complete order isomorphism} provided that $\varphi$ is a bijection and $\varphi^{-1}:\cT \to \cS$ is also ucp.  We will also need the notion of an \textbf{order isomorphism}, which is a linear map $\varphi:\cS \to \cT$ between operator systems that is a bijection, such that $\varphi$ and $\varphi^{-1}$ are $1$-positive.  Finally, given a linear map $\varphi:\cS \to \cT$ between operator systems, we say that $\varphi$ is a \textbf{complete order embedding} (or \textbf{complete order injection}) if there is an operator subsystem $\cT_1 \subseteq \cT$ such that $\varphi(\cS)=\cT_1$ and $\varphi:\cS \to \cT_1$ is a complete order isomorphism.

The following celebrated result of Choi and Effros ensures that there is no difference between considering abstract operator systems and concrete operator systems.

\begin{mythe}
\emph{(Choi-Effros, \cite{choieffros})}
Let $\cS$ be an abstract operator system equipped with Archimedean matrix order unit $e$.  Then there is a Hilbert space $\cH$ and a complete order embedding $\varphi:\cS \to \bofh$ such that $\varphi(e)=I_{\cH}$.  Conversely, any operator system contained in $\bofh$ is an abstract operator system with Archimedean matrix order unit $e=I_{\cH}$.
\end{mythe}

It will be useful to consider matricial states on an operator system $\cS$.  By way of notation, for each $n \in \bN$, we let $S_n(\cS)$ be the set of all ucp maps from $\cS$ into $M_n$, and we let $S_{\infty}(\cS)=\bigcup_{n \in \bN} S_n(\cS)$.  We will often use the term \textbf{state} to refer to elements of $S_1(\cS)$.

Any operator system $\cS$ has a sequence of matrix norms which, when completed, give $\cS$ an operator space structure.  Indeed, if $(\cS,\{C_n\}_{n=1}^{\infty},e)$ is the operator system structure on $\cS$ and $X \in M_n(\cS)$, the norms $$\|X\|_n=\inf \left\{r>0: \begin{pmatrix} rI_n & X \\ X^* & rI_n\end{pmatrix} \in C_{2n} \right\}$$
make $\cS$ into a matricially normed space (see, for example, \cite[Chapter 13]{paulsen02}).

Before moving on to tensor products, we require some information about duals of operator systems.  The Banach space dual of an operator system $\cS$ can be given the structure of a matrix-ordered $*$-vector space \cite[Lemma 4.2, Lemma 4.3]{choieffros}.  This structure is given as follows: let $\cS^d$ denote the Banach space dual of $\cS$.  Given $f \in \cS^d$, we define $f^* \in \cS^d$ by $f^*(x)=\overline{f(x^*)}$ for $x \in \cS$, which makes $\cS^d$ into a $*$-vector space.  We say an element $(f_{ij}) \in M_n(\cS^d)$ is positive provided that the map $F:\cS \to M_n$ given by $F(x)=(f_{ij}(x))$ is completely positive.  If we also assume that $\cS$ is finite-dimensional, then $\cS^d$ becomes an operator system with order unit given by a faithful state on $\cS$ \cite{choieffros}.  In fact, if $\cS$ is finite-dimensional, then $\cS^{dd}$ and $\cS$ are completely order isomorphic via the canonical map $i:\cS \to \cS^{dd}$.

\begin{mydef}
\emph{(Kavruk-Paulsen-Todorov-Tomforde, \cite{KPTT})}
Let $\cS$ and $\cT$ be operator systems.  A collection of matricial cones $\tau=\{C_n\}_{n=1}^{\infty}$ with $C_n \subseteq M_n(\cS \otimes \cT)_h$ is said to be an \textbf{operator system structure} on $\cS \otimes \cT$ if
\begin{enumerate}
\item
$(\cS \otimes \cT,\{C_n\}_{n=1}^{\infty},1_{\cS} \otimes 1_{\cT})$ is an operator system,
\item
$(s_{ij} \otimes t_{k\ell}) \in C_{nm}$ whenever $(s_{ij}) \in M_n(\cS)_+$, $(t_{k\ell}) \in M_m(\cT)_+$ and $n,m \in \bN$, and
\item
whenever $n,k \in \bN$ and $\varphi \in S_n(\cS)$ and $\psi \in S_k(\cT)$, then $\varphi \otimes \psi \in S_{nk}(\cS \otimes \cT)$ with respect to the collection of matricial cones $\tau=\{C_n\}_{n=1}^{\infty}$.
\end{enumerate}
We denote by $\cS \otimes_{\tau} \cT$ the resulting operator system.
\end{mydef}

Let $\mathcal{O}$ be the category of operator systems with ucp maps as the morphisms.  Following the definitions in \cite{KPTT}, we say that a mapping $\tau:\mathcal{O} \times \mathcal{O} \to \mathcal{O}$ given by $(\cS,\cT) \mapsto \cS \otimes_{\tau} \cT$ is an \textbf{operator system tensor product} if for all operator systems $\cS$ and $\cT$, the matrix ordering on $\tau(\cS,\cT)$ is an operator system structure on $\cS \otimes \cT$.  We say that an operator system tensor product $\tau$ is \textbf{functorial} if it satisfies the following property:

\begin{itemize}
\item
If $\cS_1$ and $\cT_1$ are operator systems and $\varphi:\cS \to \cS_1$ and $\psi:\cT \to \cT_1$ are ucp maps, then $\varphi \otimes \psi: \cS \otimes_{\tau} \cS_1 \to \cT \otimes_{\tau} \cT_1$ is ucp.
\end{itemize}

\begin{mydef}
\emph{(Kavruk-Paulsen-Todorov-Tomforde, \cite{KPTT})}
Let $\cS$ and $\cT$ be operator systems.  The \textbf{minimal tensor product} of $\cS$ and $\cT$ is the vector space $\cS \otimes \cT$, with order unit $1 \otimes 1$, equipped with positive cones in $M_n(\cS \otimes \cT)$ given by the set $C_n^{\min}(\cS,\cT)$ of all $X \in M_n(\cS \otimes \cT)$ for which $X=X^*$ and $\varphi \otimes \psi(X) \geq 0$ whenever $\varphi \in S_{\infty}(\cS)$ and $\psi \in S_{\infty}(\cT)$.
\end{mydef}

\begin{mydef}
\emph{(Kavruk-Paulsen-Todorov-Tomforde, \cite{KPTT})}
Given operator systems $\cS,\cT$, the \textbf{commuting tensor product} of $\cS$ and $\cT$ is the vector space $\cS \otimes \cT$ with order unit $1 \otimes 1$, with positive cones $C_n^{\text{comm}}(\cS,\cT)$ given by the following property: $X \in M_n(\cS \otimes \cT)_h$ is in $C_n^{\text{comm}}(\cS,\cT)$ if and only if whenever $\cH$ is a Hilbert space and $\varphi:\cS \to \bofh$ and $\psi:\cT \to \bofh$ are ucp maps with commuting ranges, then $\varphi \cdot \psi(X) \geq 0$, where $\varphi \cdot \psi(s \otimes t):=\varphi(s)\psi(t)$ for all $s \in \cS$ and $t \in \cT$.
\end{mydef}

\begin{mydef}
\emph{(Kavruk-Paulsen-Todorov-Tomforde, \cite{KPTT})}
Let $\cS,\cT$ be operator systems.  For each $n \in \bN$, define the set $D_n^{\max}(\cS,\cT)$ to be the set of all $X \in M_n(\cS \otimes \cT)_h$ for which there exist $S \in M_k(\cS)_+$, $T \in M_m(\cT)_+$ and a linear map $A:\bC^k \otimes \bC^m \to \bC^n$ such that $$X=A(S \otimes T)A^*.$$
Then the \textbf{maximal tensor product} of $\cS$ and $\cT$ is defined to be the operator system $(\cS \otimes \cT,1 \otimes 1,C_n^{\max}(\cS,\cT))$, where $$C_n^{\max}(\cS,\cT)=\{ X \in M_n(\cS \otimes \cT)_h: \forall \ee>0, X+\ee 1 \in D_n^{\max}(\cS,\cT)\}.$$
\end{mydef}

Each of $\min$, $c$ and $\max$ are functorial operator system tensor products \cite{KPTT}.  For finite-dimensional operator systems, the min and max tensor products are dual to each other.

\begin{pro}
\emph{(Farenick-Paulsen, \cite{FP})}
\label{dualofminmax}
If $\cS$ and $\cT$ are finite-dimensional operator systems, then $(\cS \otimes_{\min} \cT)^d$ is completely order isomorphic to $\cS^d \otimes_{\max} \cT^d$, and $(\cS \otimes_{\max} \cT)^d$ is completely order isomorphic to $\cS^d \otimes_{\min} \cT^d$.
\end{pro}

Two more tensor products are of interest: the \textbf{essential left} and \textbf{essential right} tensor products.

\begin{mydef}
\emph{(Kavruk-Paulsen-Todorov-Tomforde, \cite{KPTT})}
Let $\cS,\cT$ be operator systems.  Define the operator system $\cS \otimes_{\el} \cT$ to be the operator system structure arising from the inclusion $\cS \otimes \cT \subseteq \cI(\cS) \otimes_{\max} \cT$, where $\cI(\cS)$ is the injective envelope of $\cS$.  (See \cite{hamana} or \cite[Chapter 15]{paulsen02} for more on injective envelopes.)  Similarly, define $\cS \otimes_{\er} \cT$ to be the operator system structure arising from the inclusion $\cS \otimes \cT \subseteq \cS \otimes_{\max} \cI(\cT)$.
\end{mydef}

The tensor products $\er$ and $\el$ are examples of asymmetric tensor products; in fact, the map $s \otimes t \mapsto t \otimes s$ induces a complete order isomorphism $\cS \otimes_{\er} \cT \simeq \cT \otimes_{\el} \cS$ and $\cS \otimes_{\el} \cT \simeq \cT \otimes_{\er} \cS$.

Given two operator system tensor products $\alpha,\beta$, we will write $\alpha \leq \beta$ to mean that whenever $\cS,\cT$ are operator systems, the identity map $\id:\cS \otimes_{\beta} \cT \to \cS \otimes_{\alpha} \cT$ is ucp.  For example, we have the following (see \cite{KPTT}): $$\min \leq \er, \, \el \leq c \leq \max.$$
For operator system tensor products $\alpha,\beta$, we say that an operator system $\cS$ is $(\alpha,\beta)$-nuclear if $\cS \otimes_{\alpha} \cT=\cS \otimes_{\beta} \cT$ for all operator systems $\cT$.  Equivalently, $\cS$ is $(\alpha,\beta)$-nuclear if the identity map $\id:\cS \otimes_{\alpha} \cT \to \cS \otimes_{\beta} \cT$ is a complete order isomorphism for every operator system $\cT$.

Frequently, the theory of operator system tensor products has been motivated by the theory of operator space tensor products.  In particular, suppose that $X$ is an operator space contained in $\bofh$.  Then there is a canonical operator system that contains a completely isometric copy of $X$; namely, one may define $$\cS_X=\left\{ \begin{pmatrix} \lambda I_{\cH} & x \\ y^* & \mu I_{\cH} \end{pmatrix} \in M_2(\bofh): x,y \in X, \, \lambda,\mu \in \bC \right\},$$
which is equipped with the complete isometry $X \hookrightarrow \cS_X$ given by sending $x \in X$ to the matrix $\begin{pmatrix} 0 & x \\ 0 & 0 \end{pmatrix}$. See \cite[Chapter 8]{paulsen02} for more information on $\cS_X$.  It is left to the reader to check that $\cS_X$ does not depend on the embedding $X \subseteq \bofh$, up to unital complete order isomorphism.  For operator spaces $X$ and $Y$, considering the copy of $X \otimes Y$ inside of $\cS_X \otimes \cS_Y$ gives rise to operator space tensor products.  For any operator spaces $X,Y$ and operator system tensor product $\tau$, there is a natural operator space tensor product, denoted by $X \otimes^{\tau} Y$, given by the inclusion of $X \otimes Y$ in $\cS_X \otimes_{\tau} \cS_Y$.  We will refer to the tensor product $X \otimes^{\tau} Y$ as the \textbf{induced operator space tensor product} of $X$ and $Y$ (see \cite{KPTT}).  Two important operator space tensor products are the injective tensor product and the projective tensor product, which we will denote by $X \check{\otimes} Y$ and $X \widehat{\otimes} Y$, respectively (see \cite{blecherpaulsen}).  The relation between tensor products of operator systems of the form $\cS_X$ and operator spaces is outlined in the following theorem.

\begin{mythe}
\emph{(Kavruk-Paulsen-Todorov-Tomforde, \cite{KPTT})}
\label{canonicalopsys}
Let $X,Y$ be operator spaces.  The following are true:
\begin{enumerate}
\item
$X \otimes^{\min} Y=X \check{\otimes} Y$ completely isometrically.
\item
$X \otimes^{\max} Y=X \widehat{\otimes} Y$ completely isometrically.
\end{enumerate}
\end{mythe}

Similar to the operator system $\cS_X$, one may define another ``canonical" operator system of an operator space $X$ by letting $$\cS_X^0=\left\{ \begin{pmatrix} \lambda I_{\cH} & x \\ y^* & \lambda I_{\cH} \end{pmatrix} \in M_2(\bofh): x,y \in X, \, \lambda \in \bC \right\}.$$
If $\tau$ is an operator system tensor product and $X,Y$ are operator spaces, then we shall denote by $X \otimes_0^{\tau} Y$ the operator space structure on $X \otimes Y$ induced by the inclusion $X \otimes Y \subseteq \cS_X^0 \otimes_{\tau} \cS_Y^0$.  The analogous result to Theorem \ref{canonicalopsys} holds for $\cS_X^0$ as well.  For completeness, we include the proofs.

\begin{lem}
Let $X,Y$ be operator spaces.  Then the inclusion $X \otimes Y \subseteq \cS_X^0 \otimes \cS_Y^0$ gives rise to an operator space tensor product of $X$ and $Y$; that is, the following conditions hold (in the sense of \cite{blecherpaulsen}):
\begin{enumerate}
\item
If $x \in M_n(X)$ and $y \in M_m(Y)$, then $$\|x \otimes y\|_{M_{nm}(X \otimes_0^{\tau} Y)} \leq \|x\|_{M_n(X)} \|y\|_{M_m(Y)}.$$
\item
If $\phi:X \to M_n$ and $\psi:Y \to M_m$ are completely bounded maps, then $\phi \otimes \psi: X \otimes_0^{\tau} Y \to M_{mn}$ is completely bounded, with $\| \phi \otimes \psi \|_{cb} \leq \|\phi \|_{cb} \|\phi \|_{cb}$.
\end{enumerate}
\end{lem}

\begin{proof}
By \cite[Proposition 3.4]{KPTT}, treating $\cS_X^0 \otimes_{\tau} \cS_Y^0$ as an operator space yields an operator space tensor product of the operator spaces $\cS_X^0$ and $\cS_Y^0$.  Since the inclusions $X \subseteq \cS_X^0$ and $Y \subseteq \cS_Y^0$ are complete isometries, condition (1) follows since it holds for the operator space tensor product $\cS_X^0 \otimes_{\tau} \cS_Y^0$.

To show that condition (2) holds, we may assume without loss of generality that $\phi$ and $\psi$ are completely contractive.  Let $\Phi:\cS_X^0 \to M_2(M_n)$ be defined by $$\Phi \left( \begin{pmatrix} \lambda & x_1 \\ x_2^* & \lambda \end{pmatrix} \right)=\begin{pmatrix} \lambda I_n & \phi(x_1) \\ \phi(x_2)^* & \lambda I_n \end{pmatrix}.$$
The proof that $\Phi$ is ucp is standard and analogous to \cite[Lemma 8.1]{paulsen02}.  Similarly, the map $\Psi:\cS_Y^0 \to M_2(M_m)$ given by $$\Psi \left( \begin{pmatrix} \lambda & y_1 \\ y_2^* & \lambda \end{pmatrix} \right)=\begin{pmatrix} \lambda I_m & \psi(y_1) \\ \psi(y_2)^* & \lambda I_m \end{pmatrix}$$
is unital and completely positive.  By Property (3) of operator system tensor products, $\Phi \otimes \Psi: \cS_X^0 \otimes_{\tau} \cS_Y^0 \to M_{4mn}$ is ucp.  Compressing to a corner block yields $\phi \otimes \psi$, so that $\|\phi \otimes \psi \|_{cb} \leq 1$.
\end{proof}

Since the minimal operator system tensor product and the spatial operator space tensor product are injective (and equal for operator systems), the following result is immediate.

\begin{mythe}
\label{canonicalmin}
Let $X$ and $Y$ be operator spaces.  Then $X \otimes_0^{\min} Y$ is completely isometric to $X \check{\otimes} Y$.
\end{mythe}

The analogous result also holds for the maximal operator system tensor product.

\begin{mythe}
\label{canonicalmax}
Let $X$ and $Y$ be operator spaces.  Then $X \otimes_0^{\max} Y$ is completely isometric to $X \widehat{\otimes} Y$.
\end{mythe}

\begin{proof}
Let $U \in M_p(X \otimes Y)$.  Define $\|U\|^{\max}$ to be the norm of $U$ in $\cS_X^0 \otimes_{\max} \cS_Y^0$, and define $\|U\|$ to be the norm of $U$ in $X \widehat{\otimes} Y$.  First, suppose that $\|U\|<1$.  Every operator system tensor product is an operator space tensor product \cite[Proposition 3.4]{KPTT}, so that $\| \cdot \|^{\max}$ is smaller than the projective tensor product norm \cite{blecherpaulsen}.  Hence, $\|U\|^{\max}<1$.

Conversely, suppose that $\|U\|^{\max}<1$.  As in the proof of \cite[Theorem 5.9]{KPTT}, let $e=\begin{pmatrix} e_1 & 0 \\ 0 & e_2 \end{pmatrix}$ and $f=\begin{pmatrix} f_1 & 0 \\ 0 & f_2 \end{pmatrix}$ be the identities of $\cS^0_X$ and $\cS^0_Y$, respectively; then $e \otimes f$ is the identity of $\cS^0_X \otimes_{\max} \cS^0_Y$.  Write $U=(u_{r,s}) \in M_p(X \otimes Y)$.  Then $$\begin{pmatrix} \|U\|^{\max} (e \otimes f)_p & U \\ U^* & \|U\|^{\max}(e \otimes f)_p \end{pmatrix} \in C_{2p}^{\max}(\cS_X^0,\cS_Y^0).$$
By adding $(1-\|U\|^{\max}) \begin{pmatrix} (e \otimes f)_p & 0 \\ 0 & (e \otimes f)_p \end{pmatrix}$, it follows that $$\begin{pmatrix} (e \otimes f)_p & U \\ U^* & (e \otimes f)_p \end{pmatrix} \in D_{2p}^{\max}(\cS_X^0,\cS_Y^0).$$
This implies that there are matrices $P=(P_{ij}) \in M_n(\cS_X^0)_+$, $Q=(Q_{k \ell}) \in M_m(\cS_Y^0)_+$ and $T=\begin{pmatrix} A \\ B \end{pmatrix}$ where $A=(a_{r,(i,k)})$ and $B=(b_{r,(i,k)})$ are $p \times mn$ matrices of scalars, such that $$\begin{pmatrix} (e \otimes f)_p & U \\ U^* & (e \otimes f)_p \end{pmatrix}=T(P \otimes Q)T^*.$$
Comparing blocks yields the equations $(e \otimes f)_p=A(P \otimes Q)A^*$, $U=A(P \otimes Q)B^*$, $U^*=B(P \otimes Q)A^*$ and $(e \otimes f)_p=B(P \otimes Q)B^*$.  We may write $P_{ij}=\begin{pmatrix} \alpha_{ij}e_1 & x_{ij} \\ w_{ij}^* & \alpha_{ij}e_2 \end{pmatrix}$ and $Q_{k\ell}=\begin{pmatrix} \gamma_{k\ell}f_1 & y_{k\ell} \\ z_{k\ell}^* & \gamma_{k\ell}f_2 \end{pmatrix}$ where $\alpha_{ij},\gamma_{k\ell} \in \bC$, $x_{ij},w_{ij} \in X$, and $y_{k\ell},z_{k\ell} \in Y$.  Now set $R=(\alpha_{ij})$, $S=(\gamma_{k\ell})$, $\cX=(x_{ij})$ and $\cY=(y_{k\ell})$.

The fact that $P,Q$ are positive implies that $R$ and $S$ are positive, $(w_{ij}^*)=\cX^*$, $(z_{k\ell}^*)=\cY^*$, and that for every $r>0$, we have $\|(R+rI_n)^{-1/2} \cX (R+rI_n)^{-1/2}\| \leq 1$ in $M_n(X)$ and $\|(S+rI_m)^{-1/2}\cY(S+rI_m)^{-1/2}\| \leq 1$ in $M_n(Y)$ \cite[p.~99]{paulsen02}.

Let $Re_1:=(\alpha_{ij}e_1)$, and similarly define $Re_2$, $Sf_1$ and $Sf_2$.  Looking at the blocks of the $4 \times 4$ block matrix equation $(e \otimes f)_p=A(P \otimes Q)A^*$, we obtain the equations $(e_i \otimes f_j)_p=A(Re_i \otimes Sf_j)A^*$ for $i,j=1,2$.  Therefore, $I_p=A(R \otimes S)A^*=B(R \otimes S)B^*$.  The element $U$ is only present in the $(1,4)$-block of the $4 \times 4$ block matrix, with the other blocks being equal to zero.  Hence, the equation $U=A(P \otimes Q)B^*$ in $\cS_X^0 \otimes \cS_Y^0$ gives $U=A(\cX \otimes \cY)B^*$ in $X \otimes Y$.  If $R,S$ are invertible, then let $A_0=A(R \otimes S)^{\frac{1}{2}}$ and let $B_0=B(R \otimes S)^{\frac{1}{2}}$.  Then $$U=A_0(R \otimes S)^{-1/2} (\cX \otimes \cY)(R \otimes S)^{-1/2} B_0^*=A_0[(R^{-1/2}\cX R^{-1/2}) \otimes (S^{-1/2} \cY S^{-\frac{1}{2}})]B_0^*.$$  We know that $A_0A_0^*=B_0B_0^*=I_p$.  Thus, letting $\cX_0=R^{-1/2}\cX R^{-1/2}$ and $\cY_0=S^{-1/2}\cY S^{-1/2}$, we have $U=A_0(\cX_0 \otimes \cY_0)B_0^*$ with all the matrices appearing in this factorization having norm at most one.  Therefore, $\|U\| \leq 1$.

If either of $R$ or $S$ are not invertible, then add $rI_n$ and $rI_m$ to $R$ and $S$, respectively, for $r>0$, and define new matrices $A_0=A[(R+rI_n) \otimes (S+rI_m)]^{1/2}$ and $B_0=B[(R+rI_n) \otimes (S+rI_m)]^{-1/2}$.  The corresponding factorization yields $\|U\| \leq 1+Cr$ for some $C$ that is independent of $r$.  Since this is possible for all $r>0$, we obtain $\|U\| \leq 1$.
\end{proof}

\section{The OSLLP, WEP AND DCEP for Operator Systems}

There are two $C^{\ast}$-algebraic properties that play a crucial role in Kirchberg's conjecture: the local lifting property (LLP) and the weak expectation property (WEP) (see \cite{kirchberg93}).  Here we outline these properties for operator systems, as well as the double commutant expectation property (DCEP).  For the convenience of the reader, we give some known characterizations of these properties in terms of tensor products with $\bofh$ and tensor products with $C^*(F_{\infty})$.  See \cite{quotients} for more information and proofs of the theorems in this section.

Let $\cS$ be an operator system.  We say that $\cS$ has the \textbf{operator system local lifting property} (OSLLP) if whenever $\cA$ is a unital $C^{\ast}$-algebra, $\cI \subseteq \cA$ is a two-sided ideal with $\pi:\cA \to \cA/\cI$ the canonical quotient map, and $\varphi:\cS \to \cA/\cI$ is a ucp map, then for every finite-dimensional operator system $\cT \subseteq \cS$, there is a ucp map $\psi_{\cT}:\cT \to \cA$ such that $\pi \circ \psi_{\cT}=\varphi_{|\cT}$.  This property for operator systems was initially defined in \cite{quotients}. The OSLLP can be characterized in a few ways.

\begin{mythe}
\emph{(Kavruk-Paulsen-Todorov-Tomforde, \cite{quotients})}
\label{osllp}
Let $\cS$ be an operator system.  The following are equivalent.
\begin{enumerate}
\item
$\cS$ has the OSLLP.
\item
$\cS$ is $(\min,\er)$-nuclear.
\item
$\cS \otimes_{\min} \bofh=\cS \otimes_{\max} \bofh$ for all Hilbert spaces $\cH$.
\end{enumerate}
\end{mythe}

Also defined in \cite{quotients}, we say that an operator system $\cS$ has the \textbf{weak expectation property} (WEP) if the canonical inclusion $i:\cS \hookrightarrow \cS^{dd}$ extends to a ucp map $\phi:\cI(\cS) \to \cS^{dd}$, where $\cI(\cS)$ is the injective envelope of $\cS$.  The WEP is in fact a form of nuclearity.

\begin{mythe}
\emph{(Kavruk-Paulsen-Todorov-Tomforde, \cite{quotients})}
\label{wepnuclearity}
An operator system $\cS$ has the WEP if and only if $\cS$ is $(\el,\max)$-nuclear.
\end{mythe}

Finally, we say that an operator system $\cS$ has the \textbf{double commutant expectation property} \cite{quotients} (DCEP), if for any unital complete order embedding $\iota:\cS \to \bofh$, there is a ucp extension $\varphi:\cI(\cS) \to \iota(\cS)''$ of $\iota$, where $\cI(\cS)$ is the injective envelope of $\cS$ and $\iota(\cS)''$ is the double commutant of $\iota(\cS)$ in $\bofh$.  The DCEP is a weaker form of nuclearity than the WEP.

\begin{mythe}
\emph{(Kavruk-Paulsen-Todorov-Tomforde, \cite{quotients})}
\label{dcep}
Let $\cS$ be an operator system.  The following are equivalent.
\begin{enumerate}
\item
$\cS$ has the DCEP.
\item
$\cS$ is $(\el,c)$-nuclear.
\item
$\cS \otimes_{\min} C^*(F_{\infty})=\cS \otimes_{\max} C^*(F_{\infty})$.
\end{enumerate}
\end{mythe}

All unital $C^{\ast}$-algebras are $(c,\max)$-nuclear \cite[Theorem 6.7]{KPTT}, so that the WEP and the DCEP are the same for unital $C^{\ast}$-algebras.  However, the DCEP is a weaker property in general.  Indeed, the operator system $\cT_n$ of tridiagonal matrices in $M_n$ is $(\min,c)$-nuclear but $\cT_n \otimes_{\min} \cT_n^d \neq \cT_n \otimes_{\max} \cT_n^d$ for $n \geq 3$ (see \cite{KPTT} for more information on $\cT_n$).  Thus, $\cT_n$ has the DCEP but not the WEP.

\section{Complete Quotient Maps}

The theory of operator system quotients is not as straightforward as in other categories.  Here, we give a very brief introduction to this quotient theory.  Much more information on operator system quotients can be found in \cite{quotients}.  Suppose that $\cS,\cT$ are operator systems and $\varphi:\cS \to \cT$ is a ucp map with kernel $\cJ$.  We endow the quotient vector space $\cS/\cJ$ with an operator system structure as follows.  If $q:\cS \to \cS/\cJ$ denotes the canonical (vector space) quotient map and we denote by $\dot{x}$ the image $q(x)$ of a vector $x \in \cS$, then the order unit of $\cS/\cJ$ is $\dot{1}$, while the adjoint of $\dot{x}$ is simply $\dot{x^*}$.  We define the sets $$D_n(\cS,\cJ)=\{ \dot{X} \in M_n(\cS/\cJ)_h: \exists Y \in M_n(\cS)_+ \text{ such that } q^{(n)}(Y)=\dot{X}\},$$
where $q^{(n)}$ is the $n$-fold amplification of $q$.  To ensure that the positive cones on $\cS/\cJ$ satisfy the Archimedean property, we define the cones to be $$C_n(\cS,\cJ)=\{\dot{X} \in M_n(\cS/\cJ)_h: \forall \ee>0, \, \dot{X}+\ee 1 \in D_n(\cS,\cJ)\}.$$
In general, the norm induced by the above operator system structure on $\cS/\cJ$ can differ greatly from the usual quotient norm on $\cS/\cJ$ (see \cite[Example 4.4]{quotients} or \cite[Lemma 2.1]{FP} for examples).
 
Given operator systems $\cS,\cT$ and a u.c.p. map $\varphi:\cS \to \cT$ with kernel $\cJ$, we say that $\cJ$ is \textbf{completely order proximinal} if $D_n(\cS,\cJ)=C_n(\cS,\cJ)$ for all $n$. A surjective ucp map $\varphi:\cS \to \cT$ with kernel $\cJ$ is said to be a \textbf{complete quotient map} if the induced map $\dot{\varphi}:\cS/\cJ \to \cT$ given by $\dot{\varphi}(\dot{x})=\varphi(x)$ is a complete order isomorphism.  There is a relation between complete quotient maps and complete order injections via adjoint maps.  For a linear map $\varphi:\cS \to \cT$, we define the adjoint map $\varphi^d:\cT^d \to \cS^d$ by $[\varphi^d(\psi)](s)=\psi(\varphi(s))$ for all $\psi \in \cT^d$ and $s \in \cS$.

\begin{pro}
\emph{(Farenick-Paulsen, \cite{FP})}
\label{completequotientdual}
Let $\cS$ and $\cT$ be finite-dimensional operator systems.  Then a linear map $\varphi:\cS \to \cT$ is a complete quotient map if and only if $\varphi^d:\cT^d \to \cS^d$ is a complete order injection.
\end{pro}

\section{The $C^{\ast}$-algebra $\cU_{nc}(n)$}

For any $n \in \bN$, denote by $\cU_{nc}(n)$ the universal $C^{\ast}$-algebra of $n^2$ generators $\{u_{ij}\}_{i,j=1}^n$ with the restriction that the matrix $U=(u_{ij})$ is unitary.  This algebra was first defined by L. Brown in \cite{brown}.  We may define the operator subsystem $$\cV_n=\spn (\{1\} \cup \{u_{ij}\}_{i,j=1}^n \cup \{u_{ij}^*\}_{i,j=1}^n).$$

The operator system possesses an important universal property.

\begin{pro}
\label{univprop}
Let $(T_{ij}) \in M_n(\bofh)$ be a contraction.  Then there is a unique ucp map $\psi:\cV_n \to \bofh$ such that $\psi(u_{ij})=T_{ij}$ for all $1 \leq i,j \leq n$.  The ucp map $\psi$ dilates to a unital $*$-homomorphism $\pi:\cU_{nc}(n) \to M_2(\bofh)$ such that $\pi(u_{ij})=\begin{pmatrix} T_{ij} & (\sqrt{I-TT^*})_{ij} \\ (\sqrt{I-T^*T})_{ij} & -T_{ji}^* \end{pmatrix}$ for all $1 \leq i,j \leq n$.
\end{pro}

\begin{proof}
Let $T=(T_{ij})$; then $\|T\| \leq 1$.  The operator $V=\begin{pmatrix} T & \sqrt{I-TT^*} \\ \sqrt{I-T^*T} & -T^* \end{pmatrix}$ is unitary in $M_2(M_n(\bofh))$.  Performing a canonical shuffle (see \cite[p.~97]{paulsen02}) yields a unitary $W=(W_{ij}) \in M_n(M_2(\bofh))$ such that $$W_{ij}=\begin{pmatrix} T_{ij} & (\sqrt{I-TT^*})_{ij} \\ (\sqrt{I-T^*T})_{ij} & -T_{ji}^* \end{pmatrix}.$$
By the universal property of $\cU_{nc}(n)$, there is a unital $*$-homomorphism $\pi:\cU_{nc}(n) \to M_2(\bofh)$ with $\pi(u_{ij})=W_{ij}$ for all $i,j$.  Compressing to the $(1,1)$-corner in $M_2(\bofh)$ yields the ucp map $\psi$, as desired.
\end{proof}

We remark that $\cV_n$ is the image of a unital, completely positive map on $M_{2n}$.  Indeed, define $\varphi:M_{2n} \to V_n$ by $$\varphi(E_{ij})=\begin{cases} \frac{1}{2n} 1 \text{ if } i=j \\ \frac{1}{2n} u_{i,j-n} \text{ if } i \leq n \text{ and } j \geq n+1 \\ \frac{1}{2n}u_{j,i-n}^* \text{ if } i \geq n+1 \text{ and } j \leq n \\ 0 \text{ otherwise}. \end{cases}$$
Then the Choi matrix of $\varphi$ is $(\varphi(E_{ij}))=\begin{pmatrix} \frac{1}{2n} I & \frac{1}{2n} U \\ \frac{1}{2n} U^* & \frac{1}{2n} I \end{pmatrix}$.  Since $U^*U=I$, $(\varphi(E_{ij})) \in M_2(M_n(\cV_n))_+$, so that $\varphi$ is unital and completely positive by a theorem of Choi (see \cite[Theorem 3.14]{paulsen02}).  We will denote by $\cJ_{2n}$ the kernel of $\varphi$.  Then $\cJ_{2n}=\left\{ \begin{pmatrix} A & 0 \\ 0 & B \end{pmatrix} \in M_2(M_n): \tr(A \oplus B)=0 \right\}$.  It is readily checked that $\cJ_{2n}$ contains no positive element except $0$.  It follows by \cite[Proposition 2.4]{kavruknuclearity} that $\cJ_{2n}$ is completely order proximinal.

To simplify notation, we define for each $n \geq 2$ the index sets $\Lambda_n^+=\{(i,j) \in \{1,...,2n\}^2: i \leq n, \, j \geq n+1\}$ and $\Lambda_n^-=\{(i,j) \in \{1,...,2n\}^2: i \geq n+1, \, j \leq n\}$.  We define $\Lambda_n=\Lambda_n^+ \cup \Lambda_n^-$.  Using the notation of \cite{FP}, whenever $i,j \in \{1,...,2n\}$ are such that $\varphi(E_{ij}) \neq 0$, we define $e_{ij}=\dot{E}_{ij} \in M_{2n}/\cJ_{2n}$.  We let $q:M_{2n} \to M_{2n}/\cJ_{2n}$ be the canonical quotient map.  To show that $\cV_n$ is a complete quotient of $M_{2n}$, we need some equivalent characterizations of positivity in the quotient operator system $M_{2n}/\cJ_{2n}$.  This result and its proof are analogous to \cite[Proposition 2.3]{FP}.  The key difference is the use of the universal property of $\cV_n$ in the proof that (6) implies (5) below.

\begin{lem}
\label{positivity}
Let $A_{11},A_{ij} \in M_p$ for every $(i,j) \in \Lambda_n$.  The following are equivalent.
\begin{enumerate}
\item
$\dot{1} \otimes A_{11}+\sum_{(i,j) \in \Lambda_n} e_{ij} \otimes A_{ij}$ is positive in $(M_{2n}/\cJ_{2n}) \otimes M_p$.
\item
$\dot{1} \otimes A_{11}+\sum_{(i,j) \in \Lambda_n} \dot{\psi}(e_{ij}) \otimes A_{ij}$ is positive in $M_r \otimes M_p$ whenever $r \in \bN$ and $\dot{\psi}:M_{2n}/\cJ_{2n} \to M_r$ is ucp.
\item
$\dot{1} \otimes A_{11} + \sum_{(i,j) \in \Lambda_n} \psi(E_{ij}) \otimes A_{ij}$ is positive in $M_r \otimes M_p$ whenever $r \in \bN$ and $\psi:M_{2n} \to M_r$ is ucp with $\psi(\cJ_{2n})=\{0\}$.
\item
Whenever $B_{ij} \in M_r$ for $(i,j) \in \Lambda_n^+$ and the matrix $B=(B_{i,j+n}) \in M_n(M_r)$ are such that $$\begin{pmatrix} \frac{1}{2n} I_r & B \\ B^* & \frac{1}{2n} I_r \end{pmatrix}$$ is positive in $M_2(M_n(M_r))$, then $$I_r \otimes A_{11}+\sum_{(i,j) \in \Lambda_n^+} B_{ij} \otimes A_{ij} + \sum_{(i,j) \in \Lambda_n^-} B_{ji}^* \otimes A_{ij} \in (M_r \otimes M_p)_+.$$
\item
Whenever $C_{ij} \in M_r$ for $(i,j) \in \Lambda_n^+$ and the matrix $C=(C_{i,j+n}) \in M_n(M_r)$ is such that $$\begin{pmatrix} I_r & C \\ C^* & I_r \end{pmatrix}$$ is positive in $M_2(M_n(M_r))$, then $$I_r \otimes A_{11}+\sum_{(i,j) \in \Lambda_n^+} \frac{1}{2n} C_{ij} \otimes A_{ij}+\sum_{(i,j) \in \Lambda_n^-} \frac{1}{2n} C_{ji}^* \otimes A_{ij} \in (M_r \otimes M_p)_+.$$
\item
$I_r \otimes A_{11}+\sum_{(i,j) \in \Lambda_n^+} \frac{1}{2n} u_{i,j-n} \otimes A_{ij} + \sum_{(i,j) \in \Lambda_n^-} \frac{1}{2n} u_{j,i-n}^* \otimes A_{ij}$ is positive in $\cV_n \otimes M_p$.
\item
$I_r \otimes (2nA_{11})+\sum_{i=1}^{2n} E_{ii} \otimes B_i+\sum_{(i,j) \in \Lambda_n} E_{ij} \otimes A_{ij}$ is positive in $M_{2n} \otimes M_p$ for some matrices $B_1,...,B_{2n} \in M_p$ such that $\sum_{i=1}^{2n} B_i=(2n-4n^2)A_{11}$.
\item
$\sum_{i=1}^{2n} E_{ii} \otimes R_{ii} +\sum_{(i,j) \in \Lambda_n} E_{ij} \otimes R_{ij}$ is positive in $M_{2n} \otimes M_p$ for some matrix $R=(R_{ij}) \in (M_{2n} \otimes M_p)_+$ such that $R_{ij}=A_{ij}$ for $(i,j) \in \Lambda_n$ and $\sum_{i=1}^{2n} R_{ii}=2nA_{11}$.
\end{enumerate}
\end{lem}

\begin{proof}
Suppose that (1) holds.  Since $\cJ_{2n}$ is completely order proximinal, there exists a matrix $R=\sum_{i,j=1}^{2n} E_{ij} \otimes R_{ij} \in (M_{2n} \otimes M_p)_+$ such that $q \otimes \id_p(R)=\dot{1} \otimes A_{11}+\sum_{(i,j) \in \Lambda_n} e_{ij} \otimes A_{ij}$.  It follows that $$\dot{1} \otimes A_{11}+\sum_{(i,j) \in \Lambda_n} e_{ij} \otimes A_{ij}=q \otimes \id_p \left( \sum_{i,j=1}^{2n} E_{ij} \otimes R_{ij} \right)=\sum_{i=1}^{2n} \frac{1}{2n}e_{ii} \otimes R_{ii}+\sum_{(i,j) \in \Lambda_n} e_{ij} \otimes R_{ij}.$$
Therefore, $R_{ij}=A_{ij}$ whenever $(i,j) \in \Lambda_n$, and $\sum_{i=1}^{2n} R_{ii}=2nA_{11}$, which shows that (8) is true. \\
Assume that (8) is true, and let $R=(R_{ij}) \in (M_{2n} \otimes M_p)_+$ be as given.  Write $R_{ii}=2nA_{11}+B_i$ where $B_i=-\sum_{j \neq i} R_{jj}$.  Then $$2nA_{11}=\sum_{i=1}^{2n} R_{ii}=4n^2 A_{11}+\sum_{i=1}^{2n} B_i.$$  
Hence, $\sum_{i=1}^{2n} B_i=(2n-4n^2)A_{11}$ and (7) follows. \\
If (7) is true, then applying $\varphi \otimes \id_p$ shows that $$1 \otimes (2n A_{11})+\sum_{i=1}^{2n} \frac{1}{2n} 1 \otimes B_i+\sum_{(i,j) \in \Lambda_n^+} u_{i,j-n} \otimes A_{ij}+\sum_{(i,j) \in \Lambda_n^-} u_{j,i-n}^* \otimes A_{ij} \in (\cV_n \otimes M_p)_+.$$
Using the fact that $\sum_{i=1}^{2n} \frac{1}{2n} 1 \otimes B_i=\frac{1}{2n} 1 \otimes \left(\sum_{i=1}^{2n} B_i \right)=(1-2n) 1 \otimes A_{11}$ shows that $$1 \otimes A_{11}+\sum_{(i,j) \in \Lambda_n^+} u_{i,j-n} \otimes A_{ij}+\sum_{(i,j) \in \Lambda_n^-} u_{j,i-n}^* \otimes A_{ij} \in (\cV_n \otimes M_p)_+.$$
Thus, (7) implies (6).

Suppose that (6) is true, and let $C_{ij} \in M_r$ for $(i,j) \in \Lambda_n^+$ be such that $\begin{pmatrix} I_r & C \\ C^* & I_r \end{pmatrix} \in (M_2(M_n(M_r)))_+$, where $C=(C_{i,j+n}) \in M_n(M_r)$.  Then $C$ is a contraction in $M_n(M_r)$.  By Proposition \ref{univprop}, there is a ucp map $\psi:\cV_n \to M_r$ such that $\psi(u_{ij})=C_{i,j+n}$ for all $1 \leq i,j \leq n$.  Applying $\psi \otimes \id_p$ to the positive element in (6) shows that $$I_r \otimes A_{11}+\sum_{(i,j) \in \Lambda_n^+} \frac{1}{2n} C_{ij} \otimes A_{ij}+\sum_{(i,j) \in \Lambda_n^-} \frac{1}{2n} C_{ji}^* \otimes A_{ij} \in (M_r \otimes M_p)_+.$$
Therefore, (5) is true.

Assume (5).  Let $B_{ij} \in M_r$ for $(i,j) \in \Lambda_n^+$ and $B=(B_{i,j+n}) \in M_n(M_r)$ be such that $\begin{pmatrix} \frac{1}{2n} I_r & B \\ B^* & \frac{1}{2n} I_r \end{pmatrix}$ is in $(M_2(M_n(M_r)))_+$.  Let $C_{ij}=2nB_{ij}$, so that $B=\frac{1}{2n}C$, where $C=(C_{i,j+n})$.  Then (5) immediately implies (4).

Suppose that (4) holds, and let $\psi:M_{2n} \to M_r$ be a ucp map such that $\psi(\cJ_{2n})=\{0\}g$.  Since $E_{ii}-E_{jj} \in \cJ_{2n}$ for all $i \neq j$, we have $\psi(E_{ii})=\frac{1}{2n} I_r$ for all $1 \leq i \leq 2n$.  If $B_{ij}=\psi(E_{ij})$ for $(i,j) \in \Lambda_n^+$, then the Choi matrix of $\psi$ is $\begin{pmatrix} \frac{1}{2n} I_r & B \\ B^* & \frac{1}{2n} I_r \end{pmatrix}$, and hence must be positive.  Then (3) follows from (4).

If (3) is true and $\dot{\psi}:M_{2n}/\cJ_{2n} \to M_r$ is ucp, then $\psi:=\dot{\psi} \circ q:M_{2n} \to M_r$ is ucp and annihilates $\cJ_{2n}$.  This shows that (2) holds.

Finally, suppose that (2) is true.  Let $h=\dot{1} \otimes A_{11}+\sum_{(i,j) \in \Lambda_n} e_{ij} \otimes A_{ij}$. Note that an element $x$ of $(M_{2n}/\cJ_{2n}) \otimes M_p$ is positive if and only if whenever $r \in \bN$ and $\gamma:(M_{2n}/\cJ_{2n}) \otimes M_p \to M_r$ is ucp, then $\gamma(x) \in (M_r)_+$ (see \cite{choieffros}).  Now, if $\gamma:(M_{2n}/\cJ_{2n}) \otimes M_p \to M_r$ is ucp, then we may find linear maps $V_1,...,V_m:\bC^p \to \bC^r \otimes \bC^p$ and ucp maps $\dot{\psi}_1,...,\dot{\psi}_m:M_{2n}/\cJ_{2n} \to M_r$ such that $$\gamma=\sum_{i=1}^m V_i^* (\dot{\psi}_i \otimes \id_p(\cdot)) V_i.$$
Applying (2), we see that $\gamma(h) \in (M_r)_+$ for each ucp map $\gamma:(M_{2n}/\cJ_{2n}) \otimes M_p \to M_r$ and for each $r \in \bN$. Thus, $h$ is positive in $(M_{2n}/\cJ_{2n}) \otimes M_p$. Therefore, (1) follows from (2), as desired.
\end{proof}

\begin{mythe}
For $\varphi:M_{2n} \to \cV_n$ and for $\cJ_{2n}$ as above, the following are true:
\begin{enumerate}
\item
The map $\varphi:M_{2n} \to \cV_n$ is a complete quotient map; i.e., $M_{2n}/\cJ_{2n}$ is completely order isomorphic to $\cV_n$.
\item
The $C^{\ast}$-envelope of $\cV_n$ is $\cU_{nc}(n)$.
\end{enumerate}
\end{mythe}

\begin{proof}
The proof is similar to the proof of \cite[Theorem 2.4]{FP}. Since $\varphi$ is a surjection, the map $\dot{\varphi}:M_{2n}/\cJ_{2n} \to \cV_n$ given by $\dot{\varphi}(\dot{x})=\varphi(x)$ is ucp and a linear bijection.  Using the fact that statements (1) and (6) are equivalent in Lemma \ref{positivity}, we see that $\dot{\varphi}$ is a complete order isomorphism, which proves the first statement.

For the second statement, we will show that $\cU_{nc}(n)$ satisfies the universal property of $C_e^{\ast}(\cV_n)$ (see, for example, \cite{GL}).  Let $\cA$ be any unital $C^*$-algebra equipped with a unital complete order embedding $\iota:\cV_n \to \cA$ such that $C^*(\iota(\cV_n))=\cA$.  We assume that $\cU_{nc}(n)$ is represented faithfully on some Hilbert space $\cH$.  The identity map $\id:\cV_n \to \cV_n \subseteq \cU_{nc}(n)$ can be written as $\id=\kappa \circ \iota$, where $\kappa:\iota(\cV_n) \to \cV_n$ is the ucp inverse of $\iota$.  We extend $\kappa$ to a ucp map $\rho:\cA \to \bofh$ by Arveson's extension theorem \cite{arveson}.  Let $\rho=V^* \pi(\cdot)V$ be a minimal Stinespring representation of $\rho$ on some Hilbert space $\cH_{\pi}=\ran(V) \oplus \ran(V)^{\perp}$.  With respect to this decomposition, for all $1 \leq i,j \leq n$, we have $$\pi(\iota(u_{ij}))=\begin{pmatrix} u_{ij} & * \\ * & * \end{pmatrix}.$$
The matrix $\pi^{(n)} \circ \iota^{(n)}(U)=(\pi \circ \iota(u_{ij}))_{i,j=1}^n$, after applying the canonical shuffle, looks like $$\begin{pmatrix} U & * \\ * & * \end{pmatrix}.$$
Since $U$ is unitary and $\pi \circ \iota$ is completely contractive, the $(1,2)$ and $(2,1)$ blocks must be $0$.  By applying the inverse shuffle, it follows that for all $i,j$, we have $$\pi(\iota(u_{ij}))=\begin{pmatrix} u_{ij} & 0 \\ 0 & * \end{pmatrix}.$$
Thus, $\rho$ is multiplicative on the generators $\{ \iota(u_{ij})\}_{i,j=1}^n$ of $\cA$, so that $\rho$ is a $*$-homomorphism with $\rho(\iota(u_{ij}))=u_{ij}$ for all $i,j$.  This shows that $\rho$ is surjective from $\cA$ onto $\cU_{nc}(n)$.  By the universal property of $C^{\ast}$-envelopes, we conclude that $C_e^{\ast}(\cV_n)=\cU_{nc}(n)$.
\end{proof}

Using the fact that $M_{2n}/\cJ_{2n} \simeq \cV_n$ allows for a description of the dual of $\cV_n$.

\begin{cor}
\label{dualofvn}
The operator system dual $\cV_n^d$ of $\cV_n$ is completely order isomorphic to $\cS_{M_n}^0$.
\end{cor}

\begin{proof}
We use the same argument as in \cite[Proposition 2.7]{FP}. Since $\varphi:M_{2n} \to \cV_n$ is a complete quotient map, $\varphi^d:\cV_n^d \to M_{2n}^d$ is a complete order embedding by Proposition \ref{completequotientdual}.  If $\{\delta_{ij}\}_{i,j=1}^{2n}$ is the dual basis for $M_{2n}^d$ of the canonical basis $\{E_{ij}\}_{i,j=1}^{2n}$ for $M_{2n}$, then $M_{2n}$ is completely order isomorphic to $M_{2n}^d$ via the mapping $E_{ij} \mapsto \delta_{ij}$ \cite[Theorem 6.2]{PTT}.  Taking the unit of $M_{2n}^d$ to be the canonical normalized trace, this mapping is a unital complete order isomorphism.  It follows that the vector space dual of $M_{2n}/\cJ_{2n}$, equipped with the operator system structure inherited from $M_{2n}$, is the operator system dual of $\cV_n$.  It is not hard to see that the vector space dual of $M_{2n}/\cJ_{2n}$ is the annihilator of $\cJ_{2n}$ in $M_{2n}^d$.  Therefore, $\cV_n^d \simeq \cS_{M_n}^0$.
\end{proof}

We will now move towards an analogue of Kirchberg's Theorem for $\cU_{nc}(n)$.  Kirchberg's famous result on the full group $C^{\ast}$-algebra of the free group $F_n$, for $n \geq 2$, is that $C^*(F_n) \otimes_{\min} \bofh=C^*(F_n) \otimes_{\max} \bofh$ for every Hilbert space $\cH$.  We will show that a similar result is true when replacing $C^*(F_n)$ by $\cU_{nc}(n)$.

First, we adopt some terminology using Lemma \ref{positivity}.  We say that an operator system $\cS$ has \textbf{property} $\fV_n$ if whenever $p \in \bN$ and $S_{11},S_{ij} \in M_p(\cS)$ for $(i,j) \in \Lambda_n$ are such that $$1 \otimes S_{11}+\sum_{(i,j) \in \Lambda_n^+} \frac{1}{2n} u_{i,j-n} \otimes S_{ij}+\sum_{(i,j) \in \Lambda_n^-} \frac{1}{2n} u_{j,i-n}^* \otimes S_{ij} \in (\cU_{nc}(n) \otimes_{\min} M_p(\cS))_+,$$
then for each $\ee>0$ there exist $R_{ij}^{\ee} \in M_p(\cS)$ for $1 \leq i,j \leq 2n$ such that
\begin{itemize}
\item
The matrix $R_{\ee}=(R_{ij}^{\ee})$ is positive in $M_{2n}(M_p(\cS))$.
\item
$R_{ij}^{\ee}=S_{ij}$ for all $(i,j) \in \Lambda_n$.
\item
$\sum_{i=1}^{2n} R_{ii}^{\ee} =2n(S_{11}+\ee 1_{M_p(\cS)})$.
\end{itemize}
Equivalently, $\cS$ has property $\fV_n$ if and only if the above holds when replacing the above positive element of $\cU_{nc}(n) \otimes_{\min} M_p(\cS)$ with $$\dot{1} \otimes S_{11}+\sum_{(i,j) \in \Lambda_n} e_{ij} \otimes S_{ij} \in (M_{2n}/\cJ_{2n} \otimes_{\min} M_p(\cS))_+.$$
We will say that $\cS$ has property $\fV$ if it has property $\fV_n$ for every $n \in \bN$.  These properties were inspired by the similar notion of operator systems having property $\mathfrak{W}_{n+1}$ with regards to the operator system $\mathcal{W}_{n+1} \subseteq C^*(F_n)$ given by $\mathcal{W}_{n+1}=\spn \{w_iw_j^*: 1 \leq i,j \leq n+1 \}$, where $w_2,...,w_{n+1}$ are the generators of $F_n$ and $w_1=1$ (see \cite{FP}).  Lemma \ref{positivity} shows that $M_p$ has property $\fV$ for every $p \in \bN$.  The operator systems satisfying property $\fV_n$ are characterized in the following proposition.

\begin{pro}
\label{propertyv}
Let $\cS$ be an operator system.  Then $\cS$ has property $\fV_n$ if and only if $\cV_n \otimes_{\min} \cS=\cV_n \otimes_{\max} \cS$.
In particular, if $\cS$ is $(\min,\max)$-nuclear, then $\cS$ has property $\fV$.
\end{pro}

\begin{proof}
We proceed as in the proof of \cite[Proposition 3.3]{FP}. Let $X=(x_{k\ell}) \in M_r(\cV_n \otimes \cS)$ for $r>1$; then $$x_{k\ell}=1 \otimes s_{11}^{(k\ell)}+\sum_{(i,j) \in \Lambda_n^+} \frac{1}{2n} u_{i,j-n} \otimes s_{ij}^{(k\ell)}+\sum_{(i,j) \in \Lambda_n^-} \frac{1}{2n} u_{j,i-n}^* \otimes s_{ij}^{(k\ell)},$$
so if $S_{11}=(s_{11}^{(k\ell)})$ and $S_{ij}=(s_{ij}^{(k\ell)})$, then we obtain $$X=1 \otimes S_{11}+\sum_{(i,j) \in \Lambda_n^+} \frac{1}{2n} u_{i,j-n} \otimes S_{ij}+\sum_{(i,j) \in \Lambda_n^-} \frac{1}{2n} u_{j,i-n}^* \otimes S_{ij},$$
where $S_{11},S_{ij} \in M_r$.  
Now, $\cS$ satisfies the definition of property $\fV_n$ when $p=r$ if and only if $M_r(\cS)$ satisfies property $\fV_n$ when $p=1$.  Moreover, we have $M_r(\cV_n \otimes_{\min} \cS)=\cV_n \otimes_{\min} M_r(\cS)$ and $M_r(\cV_n \otimes_{\max} \cS)=\cV_n \otimes_{\max} M_r(\cS)$.  By replacing $\cS$ with $M_r(\cS)$ if necessary, in order to check that $\cS$ has property $\fV_n$, it suffices to show that $\cS$ satisfies the definition of property $\fV_n$ when $p=1$.

Suppose that $\cV_n \otimes_{\min} \cS=\cV_n \otimes_{\max} \cS$, and suppose that $$x:=1 \otimes S_{11}+\sum_{(i,j) \in \Lambda_n^+} \frac{1}{2n} u_{i,j-n} \otimes S_{ij}+\sum_{(i,j) \in \Lambda_n^-} \frac{1}{2n} u_{j,i-n}^* \otimes S_{ij} \in (\cU_{nc}(n) \otimes_{\min} \cS)_+.$$
Then $x \in (\cV_n \otimes_{\max} \cS)_+$.  Hence, for every $\ee>0$, $x+\ee(1 \otimes 1_{\cS}) \in D_1^{\max}(\cV_n,\cS)$.  This means that there is $V \in M_k(\cV_n)_+$, $S \in M_m(\cS)_+$ and a linear map $A:\bC^k \otimes \bC^m \to \bC$ such that $$x+\ee (1 \otimes 1_{\cS})=A(V \otimes S)A^*.$$
Since $\varphi:M_{2n} \to \cV_n$ is a complete quotient map, there is $R \in M_k(M_{2n})_+$ such that $$x+\ee (1 \otimes 1_{\cS})=A (\varphi(R) \otimes S)A^*.$$
So, with $R_{\ee}=A(R \otimes S)A^* \in M_{2n}(\cS)_+$, we have $\varphi \otimes \id_{\cS}(R_{\ee})=x+\ee(1 \otimes 1_{\cS})$.  That is to say, for each $\ee>0$, there is $R_{\ee} \in M_{2n}(\cS)_+$ such that $$x+\ee(1 \otimes 1_{\cS})=\sum_{i=1}^{2n} 1 \otimes R_{ii}^{\ee}+\sum_{(i,j) \in \Lambda_n^+} \frac{1}{2n} u_{i,j-n} \otimes R_{ij}^{\ee}+\sum_{(i,j) \in \Lambda_n^-} \frac{1}{2n} u_{j,i-n}^* \otimes R_{ij}^{\ee}.$$
Comparing coefficients with the coefficients of $x+\ee(1 \otimes 1_{\cS})$ shows that $R_{ij}^{\ee}=S_{ij}$ for $(i,j) \in \Lambda_n$ and $\frac{1}{2n} \sum_{i=1}^{2n} R_{ii}^{\ee}=S_{11}+\ee 1$.  This shows that $\cS$ has property $\fV_n$.

Conversely, suppose that $\cS$ has property $\fV_n$ and let $p \in \bN$; we must show that $\mathcal{C}_p^{\min}(\cV_n,\cS) \subseteq \mathcal{C}_p^{\max}(\cV_n,\cS)$.  As before, by replacing $\cS$ with $M_r(\cS)$ if necessary, we may assume that $p=1$.  Let $x \in (\cV_n \otimes_{\min} \cS)_+$.  Then there are $s_{11},s_{ij} \in \cS$ for $(i,j) \in \Lambda_n$ such that $$x=1 \otimes s_{11}+\sum_{(i,j) \in \Lambda_n^+} \frac{1}{2n} u_{i,j-n} \otimes s_{ij}+\sum_{(i,j) \in \Lambda_n^-} \frac{1}{2n} u_{j,i-n}^* \otimes s_{ij}.$$
Since $\cS$ has property $\fV_n$, given $\ee>0$, there are $R_{ij}^{\ee} \in \cS$ for $1 \leq i,j \leq 2n$ such that $R_{\ee}=(R_{ij}^{\ee}) \in M_n(\cS)_+$, $R_{ij}^{\ee}=s_{ij}$ for all $(i,j) \in \Lambda_n$ and $\sum_{i=1}^{2n} R_{ii}^{\ee}=2n(s_{11}+\ee 1_{\cS})$.  If $\varphi:M_{2n} \to \cV_n$ is the complete quotient map given as before, then the map $\varphi \otimes \id_{\cS}:M_{2n} \otimes_{\max} \cS \to \cV_n \otimes_{\max} \cS$ is ucp, and $$\varphi \otimes \id_{\cS}(R_{\ee})=x+\ee(1 \otimes 1_{\cS}) \in (\cV_n \otimes_{\max} \cS)_+.$$
Thus, $x+\ee(1 \otimes 1_{\cS}) \in \cD_1^{\max}(\cV_n,\cS)$ for all $\ee>0$.  Therefore, $x \in (\cV_n \otimes_{\max} \cS)_+$, which completes the proof.
\end{proof}

The next fact about tensor products of $\cV_n$ is very useful.

\begin{pro}
\label{vncommutingorderembedding}
Let $\cS$ be any operator system.  For all $n \geq 2$, the inclusion $\cV_n \otimes_c \cS \subseteq \cU_{nc}(n) \otimes_{\max} \cS$ is a complete order embedding.
\end{pro}

\begin{proof}
By \cite[Theorem 6.7]{KPTT}, $\cA \otimes_c \cS=\cA \otimes_{\max} \cS$ for every unital $C^{\ast}$-algebra $\cA$.  We must show that $C_p^{\text{comm}}(\cV_n,\cS)=C_p^{\text{comm}}(\cU_{nc}(n),\cS) \cap M_p(\cV_n \otimes \cS)$ for each $p \in \bN$.  The inclusion map $\iota_n:\cV_n \to \cU_{nc}(n)$ is ucp, so by functoriality of the commuting tensor product, $\iota_n \otimes \id_{\cS}:\cV_n \otimes_c \cS \to \cU_{nc}(n) \otimes_{\max} \cS$ is ucp.  Therefore, $C_p^{\text{comm}}(\cV_n,\cS) \subseteq C_p^{\text{comm}}(\cU_{nc}(n),\cS) \cap M_p(\cV_n \otimes \cS)$ for each $p$.

Conversely, suppose that $X \in C_p^{\text{comm}}(\cU_{nc}(n),\cS) \cap M_p(\cV_n \otimes \cS)$.  Let $\psi:\cV_n \to \bofh$ and $\gamma:\cS \to \bofh$ be ucp maps with commuting ranges.  Let $T=(T_{ij})$ where $T_{ij}=\psi(u_{ij})$.  By Proposition \ref{univprop}, the map $\psi$ dilates to a unital $*$-homomorphism $\pi:\cU_{nc}(n) \to M_2(\bofh)$ such that, for all $1 \leq i,j \leq n$, $$\pi(u_{ij})=\begin{pmatrix} T_{ij} & (\sqrt{I-TT^*})_{ij} \\ (\sqrt{I-T^*T})_{ij} & -T_{ji}^* \end{pmatrix}.$$
We extend $\psi$ to a ucp map on all of $\cU_{nc}(n)$ by letting $\psi(x)$ be given by the $(1,1)$ corner of $\pi(x)$, for each $x \in \cU_{nc}(n)$.  Define $\widetilde{\gamma}:\cS \to M_2(\bofh)$ by setting $$\widetilde{\gamma}(s)=\begin{pmatrix} \gamma(s) & 0 \\ 0 & \gamma(s) \end{pmatrix}.$$
Since $\gamma(s)$ commutes with each $T_{ij}$ and $T_{ij}^*$, we see that $\gamma(s) \otimes I_{\cH}$ commutes with $T$ and $T^*$.  Hence, $\gamma(s) \otimes I_{\cH}$ commutes with $C^*(I_{\cH},T,T^*)$, which contains $\sqrt{I-T^*T}$ and $\sqrt{I-TT^*}$.  Thus, $\gamma(s) \otimes I_{\cH}$ commutes with each block of $\pi(u_{ij})$.  It follows that the range of $\widetilde{\gamma}$ commutes with each $\pi(u_{ij})$.  Since $\pi$ is a $*$-homomorphism, the ucp maps $\pi$ and $\widetilde{\gamma}$ must have commuting ranges.  Therefore, the map $\pi \cdot \widetilde{\gamma}:\cU_{nc}(n) \otimes_c \cS \to M_2(\bofh)$ is ucp.  Compressing to the $(1,1)$ corner in $M_2(\bofh)$ yields the map $\psi \cdot \gamma$.  It follows that $(\psi \cdot \gamma)_{|\cV_n \otimes \cS}$ is ucp on the inclusion of $\cV_n \otimes \cS$ into $\cU_{nc}(n) \otimes_c \cS$.  Hence, $(\psi \cdot \gamma)^{(p)}(X) \in M_p(\bofh)_+$.  As $\psi$ and $\gamma$ were arbitrary, we conclude that $X \in C_p^{\text{comm}}(\cV_n,\cS)$.
\end{proof}

\begin{lem}
\label{vnmincunc}
Let $\cS$ be any operator system.  Then $\cV_n \otimes_{\min} \cS=\cV_n \otimes_c \cS$ if and only if $\cU_{nc}(n) \otimes_{\min} \cS=\cU_{nc}(n) \otimes_{\max} \cS$.  In particular, if $\cS$ has property $\fV_n$, then $\cU_{nc}(n) \otimes_{\min} \cS=\cU_{nc}(n) \otimes_{\max} \cS$.
\end{lem}

\begin{proof}
Suppose that $\cU_{nc}(n) \otimes_{\min} \cS=\cU_{nc}(n) \otimes_{\max} \cS$.  Since the $\min$ tensor product is injective, $\cV_n \otimes_{\min} \cS$ is completely order isomorphic to the image of $\cV_n \otimes \cS$ in $\cU_{nc}(n) \otimes_{\min} \cS$.  By Proposition \ref{vncommutingorderembedding}, $\cV_n \otimes_c \cS$ is completely order isomorphic to its image in $\cU_{nc}(n) \otimes_{\max} \cS$.  It follows that $\cV_n \otimes_{\min} \cS=\cV_n \otimes_c \cS$.

Conversely, suppose that $\cV_n \otimes_{\min} \cS=\cV_n \otimes_c \cS$.  We employ an argument analogous to the proof of \cite[Proposition 3.6]{FP}. Let $X \in M_p(\cU_{nc}(n) \otimes_{\min} \cS)_+$, and let $\psi:\cU_{nc}(n) \to \bofh$ and $\gamma:\cS \to \bofh$ be ucp maps with commuting ranges.  Let $\psi=V^*\pi(\cdot)V$ be a minimal Stinespring representation for $\psi$ on some Hilbert space $\cH_{\pi}$.  By Arveson's commutant lifting theorem \cite[Theorem 1.3.1]{arveson}, there is a unital $*$-homomorphism $\rho:(\psi(\cU_{nc}(n)))' \to (\pi(\cU_{nc}(n)))'$ such that $\rho(T)V=VT$ for all $T \in (\psi(\cU_{nc}(n)))'$. Since $\psi$ and $\gamma$ have commuting ranges, we see that $\widetilde{\gamma}=\rho \circ \gamma:\cS \to (\pi(\cU_{nc}(n)))' \subseteq \cB(\cK)$ is ucp and its range commutes with the range of $\pi$.  Since $\cV_n \otimes_{\min} \cS=\cV_n \otimes_c \cS$, the map $(\pi \cdot \widetilde{\gamma})_{|\cV_n \otimes_{\min} \cS}$ is ucp.

Since the min tensor product is injective, $\cV_n \otimes_{\min} \cS$ is completely order isomorphic to the image of $\cV_n \otimes \cS$ in $\cU_{nc}(n) \otimes_{\min} C_e^*(\cS)$.  Arveson's extension theorem \cite{arveson} guarantees existence of a ucp extension $\eta: \cU_{nc}(n) \otimes_{\min} C_e^*(\cS) \to \cB(\cK)$ of $(\pi \cdot \widetilde{\gamma})_{|\cV_n \otimes_{\min} \cS}$.  For any $1 \leq i,j \leq n$, we see that $$\eta(u_{ij} \otimes 1)=\pi\cdot \widetilde{\gamma}(u_{ij} \otimes 1)=\pi(u_{ij}).$$
Thus, $\{ u_{ij}: 1 \leq i,j \leq n \}$ is in the multiplicative domain $\cM_{\eta}$ of $\eta$.  It follows that $\cU_{nc}(n) \otimes 1 \subseteq \cM_{\eta}$.  Hence whenever $a \in \cU_{nc}(n)$ and $s \in \cS$, we obtain $$\eta(a \otimes s)=\eta(\underbrace{(a \otimes 1)}_{\in \cM_{\eta}}(1 \otimes s))=\eta(a \otimes 1)\eta(1 \otimes s)=\pi(a)\widetilde{\gamma}(s).$$
Now, the upper-left corner of $\pi(a)\widetilde{\gamma}(s)$ is $$V^*\pi(a)\widetilde{\gamma}(s)V=V^*\pi(a)\rho(\gamma(s))V=V^*\pi(a)V\gamma(s)=\psi(a)\gamma(s).$$
Using this fact, we have $$\psi \cdot \gamma(a \otimes s)=\psi(a)\gamma(s)=V^*\pi(a)\widetilde{\gamma}(s)V=V^*\eta(a \otimes s)V.$$
So, for all $z \in \cU_{nc}(n) \otimes_{\min} \cS$ we have $\psi \cdot \gamma(z)=V^*\eta(z)V$, so that $(\psi \cdot \gamma)_{|\cU_{nc}(n) \otimes_{\min} \cS}$ is ucp.  Therefore, $(\psi \cdot \gamma)^{(n)}(X) \in M_p(M_m)_+$ so that $\cU_{nc}(n) \otimes_{\min} \cS=\cU_{nc}(n) \otimes_c \cS$. \\
\end{proof}

\begin{lem}
\label{bofhpropertyv}
If $\cH$ is any Hilbert space, then $\bofh$ has property $\fV$.  Equivalently, $\cV_n$ has the OSLLP for every $n \geq 2$.
\end{lem}

\begin{proof}
The proof proceeds in a similar manner to the proof of \cite[Proposition 3.5]{FP}. Let $n \in \bN$ and let $S_{11},S_{ij} \in \bofh$ for $(i,j) \in \Lambda_n$.  Suppose that $$1 \otimes S_{11}+\sum_{(i,j) \in \Lambda_n^+} \frac{1}{2n} u_{i,j-n} \otimes S_{ij}+\sum_{(i,j) \in \Lambda_n^-} \frac{1}{2n} u_{j,i-n}^* \otimes S_{ij} \in (\cU_{nc}(n) \otimes_{\min} \bofh)_+.$$
As in the proof of Proposition \ref{propertyv}, we may assume that $p=1$.  The matrix $0 \in M_n$ is a contraction, so by Proposition \ref{univprop}, the map $\alpha:\cV_n \to \bC$ given by $\alpha(u_{ij})=0$ and $\alpha(1)=1$ extends to a ucp map on $\cU_{nc}(n)$.  Hence, $\alpha \otimes \id_{\bofh}$ is ucp on $\cU_{nc}(n) \otimes_{\min} \bofh$, which forces $S_{11} \geq 0$.  Fix $\ee>0$.  For any finite-dimensional subspace $\cM$ of $\cH$, we know that $\cB(\cM)$ has property $\fV$.  Let $P_{\cM}$ denote the orthogonal projection onto $\cM$.  Replacing $S_{11}$ with $P_{\cM} S_{11} P_{\cM}$ and $S_{ij}$ with $P_{\cM} S_{ij} P_{\cM}$, we may find $R_{ij}^{\ee,\cM} \in \cB(\cM)$ such that

\begin{itemize}
\item
$R_{\ee,\cM}:=(R_{ij}^{\ee,\cM})$ is in $(M_{2n}(\cB(\cM)))_+$,
\item
$R_{ij}^{\ee,\cM}=P_{\cM}S_{ij}P_{\cM}$ for $(i,j) \in \Lambda_n$, and
\item
$\sum_{i=1}^{2n} R_{ii}^{\ee,\cM}=2n(P_{\cM} S_{11} P_{\cM}+\ee I_{\cM})=2n(P_{\cM}S_{11}P_{\cM}+\ee P_{\cM})$.
\end{itemize}

Clearly $\sum_{i=1}^{2n} R_{ii}^{\ee,\cM} \leq 2n(S_{11}+\ee I_{\cH})$, so since each $R_{ii}^{\ee,\cM}$ is positive, the diagonal blocks of $R_{\ee,\cM}$ are bounded.  Since $\cM$ is finite-dimensional and $R_{\ee,\cM} \geq 0$, the norm of $R_{\ee,\cM}$ is given by the largest eigenvalue.  Therefore, indexing finite-dimensional subspaces of $\cH$ by inclusion, the net $(R_{\ee,\cM})_{\cM \leq \cH, \, \dim(\cM)<\infty}$ is uniformly bounded.  Let $R_{\ee}$ be a $w^*$-limit point of the net $(R_{\ee,\cM})_{\cM}$.  Then the corresponding subnet of $(P_{\cM})_{\cM}$ converges strongly to $I_{\cH}$.  It follows that if $R_{\ee}=(R_{ij}^{\ee}) \in M_n(\bofh)$, then $R_{\ee} \geq 0$, while $R_{ij}^{\ee}=S_{ij}$ for all $(i,j) \in \Lambda_n$ and $\sum_{i=1}^{2n} R_{ii}^{\ee}=2n(S_{11}+\ee I_{\cH})$.  Therefore, $\bofh$ has property $\cV_n$ for every $n \in \bN$.
\end{proof}

\begin{mythe}
\label{unckirchberg}
$\cU_{nc}(n)$ has the LLP; i.e., $\cU_{nc}(n) \otimes_{\min} \bofh=\cU_{nc}(n) \otimes_{\max} \bofh$.
\end{mythe}

\begin{proof}
By Lemma \ref{bofhpropertyv} and Proposition \ref{propertyv}, $\cV_n \otimes_{\min} \bofh=\cV_n \otimes_{\max} \bofh$.  Applying Lemma \ref{vnmincunc} gives the desired result.
\end{proof}

It should be noted that Kirchberg's Theorem for $C^*(F_n)$ follows from Theorem \ref{unckirchberg}.  To show this fact, we will need the notion of a retract of operator systems.  We will say that an operator system $\cS$ is a \textbf{retract} of an operator system $\cT$ if there are ucp maps $\psi:\cS \to \cT$ and $\chi:\cT \to \cS$ such that $\chi \circ \psi=\id_{\cS}$.

\begin{lem}
\label{retract}
Let $\cS_1,\cS_2,\cT_1,\cT_2$ be operator systems, and let $\tau_1,\tau_2 \in \{ \min,c,\max\}$.  For $i=1,2$, suppose that $\cS_i$ is a retract of $\cT_i$. If $\cT_1 \otimes_{\tau_1} \cT_2=\cT_1 \otimes_{\tau_2} \cT_2$ completely order isomorphically (respectively, order isomorphically), then $\cS_1 \otimes_{\tau_1} \cS_2=\cS_1 \otimes_{\tau_2} \cS_2$ completely order isomorphically (respectively, order isomorphically).  
\end{lem}

\begin{proof}
Since $\min \leq c \leq \max$ as operator system tensor products, we may assume that $\tau_1 \leq \tau_2$.  Since $\cS_i$ is a retract of $\cT_i$, there are ucp maps $\varphi_i:\cS_i \to \cT_i$ and $\psi_i:\cT_i \to \cS_i$ such that $\psi_i \circ \varphi_i=\id_{\cS_i}$. For each $j=1,2$, by functoriality of $\tau_j$, the maps $\varphi_1 \otimes \varphi_2:\cS_1 \otimes_{\tau_j} \cS_2 \to \cT_1 \otimes_{\tau_j} \cT_2$ and $\psi_1 \otimes \psi_2:\cT_1 \otimes_{\tau_j} \cT_2 \to \cS_1 \otimes_{\tau_j} \cS_2$ are ucp.  Moreover, the following diagram commutes:

\begin{tikzpicture}[every node/.style={midway}]
  \matrix[column sep={20em,between origins}, row sep={7em}] at (0,0) {
    \node(uncmin) {$\cT_1 \otimes_{\tau_1} \cT_2$}  ; & \node(uncmax) {$\cT_1 \otimes_{\tau_2} \cT_2$}; \\
    \node(wnmin) {$\cS_1 \otimes_{\tau_1} \cS_2$}; & \node (wnmax) {$\cS_1 \otimes_{\tau_2} \cS_2$};\\
  };
  \draw[->] (wnmin) -- (uncmin) node[anchor=east]  {$\varphi_1 \otimes \varphi_2$};
  \draw[->] (uncmin) -- (uncmax) node[anchor=south] {$\id_{\cT_1} \otimes \id_{\cT_2}$};
  \draw[->] (uncmax) -- (wnmax) node[anchor=west] {$\psi_1 \otimes \psi_2$};
  \draw[->] (wnmin) -- (wnmax) node[anchor=south] {$\id_{\cS_1} \otimes \id_{\cS_2}$};
\end{tikzpicture}

By assumption, the map $\id:\cT_1 \otimes_{\tau_1} \cT_2 \to \cT_1 \otimes_{\tau_2} \cT_2$ is completely positive (respectively, positive).  Thus, $\id:\cS_1 \otimes_{\tau_1} \cS_2 \to \cS_1 \otimes_{\tau_2} \cS_2$ is completely positive (respectively, positive).  The result follows.
\end{proof}

For the next lemma, we define the operator system $\cS_n \subseteq C^*(F_n)$ to be $\cS_n=\spn \{1,w_1,...,w_n,w_1^*,...,w_n^*\}$, where $w_1,...,w_n$ are the generators of $F_n$.

\begin{lem}
\label{uncretract}
Let $n \geq 2$.
\begin{enumerate}
\item
$C^*(F_n)$ is a retract of $\cU_{nc}(n)$.
\item
$\cS_n$ is a retract of $\cV_n$.
\end{enumerate}
\end{lem}

\begin{proof}
To prove (1), we note that $\begin{pmatrix} w_1 \\ & \ddots \\ & & w_n \end{pmatrix} \in M_n(C^*(F_n))$ is unitary.  Hence, there is a unital $*$-homomorphism $\pi:\cU_{nc}(n) \to C^*(F_n)$ such that $\pi(u_{ij})=0$ for $i \neq j$ and $\pi(u_{ii})=w_i$.  Then $\pi \otimes \id_{\cS}:\cU_{nc}(n) \otimes_{\max} \cS \to C^*(F_n) \otimes_{\max} \cS$ is ucp, while $\id:\cU_{nc}(n) \otimes_{\min} \cS \to \cU_{nc}(n) \otimes_{\max} \cS$ is ucp by Proposition \ref{propertyv} and Lemma \ref{vnmincunc}.  We let $U_1=U:=(u_{ij}) \in M_n(\cU_{nc}(n))$.  For $2 \leq i \leq n$, let $U_i$ be the conjugation of $U$ by a permutation matrix such that the $(1,1)$-entry of $U_i$ is $u_{ii}$.  Then each $U_i \in M_n(\cU_{nc}(n))$ is unitary, so by the universal property for $C^*(F_n)$, there is a unital $*$-homomorphism $\rho:C^*(F_n) \to M_n(\cU_{nc}(n))$ such that $\rho(w_i)=U_i$ for all $1 \leq i \leq n$.  Compressing to the $(1,1)$-entry in $M_n(\cU_{nc}(n))$ gives rise to a ucp map $\psi:C^*(F_n) \to \cU_{nc}(n)$ such that $\psi(w_i)=u_{ii}$ for all $1 \leq i \leq n$.  Since $\pi \circ \psi(w_i)=w_i$ and since $w_i$ is unitary, it follows that $w_i$ lies in the multiplicative domain of $\pi \circ \psi$.  Since $C^*(F_n)$ is generated by $\{w_1,...,w_n\}$, it follows that $\pi \circ \psi$ is multiplicative on $C^*(F_n)$.  The fact that $\pi \circ \psi(w_i)=w_i$ for all $i$ forces $\pi \circ \psi=\id_{C^*(F_n)}$.  Thus, (1) holds.

For (2), since $\psi(w_i)=u_{ii} \in \cV_n$ for all $1 \leq i \leq n$, we have $\psi(\cS_n) \subseteq \cV_n$.  Clearly $\pi(\cV_n)=\cS_n$.  Since $\pi \circ \psi=\id_{C^*(F_n)}$, it follows that $\cS_n$ is a retract of $\cV_n$ via the maps $\psi_{|\cS_n}:\cS_n \to \cV_n$ and $\pi_{|\cV_n}:\cV_n \to \cS_n$.
\end{proof}

\begin{mythe}
\label{propertyvandw}
Whenever $\cS$ is an operator system with property $\fV_n$, we have $C^*(F_n) \otimes_{\min} \cS=C^*(F_n) \otimes_{\max} \cS$.
\end{mythe}

\begin{proof}
By Lemma \ref{uncretract}, $C^*(F_n)$ is a retract of $\cU_{nc}(n)$.  Applying Lemma \ref{retract}, since $\cU_{nc}(n) \otimes_{\min} \cS=\cU_{nc}(n) \otimes_{\max} \cS$, it follows that $C^*(F_n) \otimes_{\min} \cS=C^*(F_n) \otimes_{\max} \cS$, which completes the proof.
\end{proof}

\begin{cor}
\emph{(Kirchberg's Theorem, \cite{kirchberg})}
\label{kirchberg}
Let $n \geq 2$.  Then $C^*(F_n)$ has the LLP.  In other words, $C^*(F_n) \otimes_{\min} \bofh=C^*(F_n) \otimes_{\max} \bofh$.
\end{cor}

Using Theorem \ref{propertyvandw}, it is possible to characterize unital $C^{\ast}$-algebras having the WEP and operator systems having the DCEP in terms of tensor products with $\cV_2$.

\begin{mythe}
\label{weppropertyv2}
Let $\cA$ be a unital $C^{\ast}$-algebra.  The following are equivalent.
\begin{enumerate}
\item
$\cA$ has the WEP.
\item
$\cA$ has property $\fV$.
\item
$\cA$ has property $\fV_2$.
\item
$\cA \otimes_{\min} \cV_2=\cA \otimes_{\max} \cV_2$.
\end{enumerate}
\end{mythe}

\begin{proof}
Clearly (2) implies (3), while (3) implies (4) by Proposition \ref{propertyv}.  Suppose that $\cA$ has the WEP.  By Theorem \ref{wepnuclearity}, $\cA$ is $(\el,\max)$-nuclear.  By Theorem \ref{osllp}, each $\cV_n$ having the OSLLP implies that each $\cV_n$ is $(\min,\er)$-nuclear.  Hence, $$\cV_n \otimes_{\min} \cA=\cV_n \otimes_{\er} \cA=\cA \otimes_{\el} \cV_n=\cA \otimes_{\max} \cV_n=\cV_n \otimes_{\max} \cA.$$
By Proposition \ref{propertyv} and the fact that $n \geq 2$ was arbitrary, we conclude that $\cA$ has property $\fV$.  This shows that (1) implies (2).

Finally, we prove that (4) implies (1).  Suppose that $\cA \otimes_{\min} \cV_2=\cA \otimes_{\max} \cV_2$.  Then by Lemma \ref{vnmincunc}, we have $\cU_{nc}(2) \otimes_{\min} \cA=\cU_{nc}(2) \otimes_{\max} \cA$.  Using Theorem \ref{propertyvandw}, we have $C^*(F_2) \otimes_{\min} \cA=C^*(F_2) \otimes_{\max} \cA$.  As $F_{\infty}$ embeds as a subgroup into $F_2$, by \cite[Proposition 8.8]{pisierbook} it follows that there are ucp maps $\Phi:C^*(F_{\infty}) \to C^*(F_2)$ and $\Psi:C^*(F_2) \to C^*(F_{\infty})$ with $\Psi \circ \Phi=\id$.  By Lemma \ref{retract}, we have $C^*(F_{\infty}) \otimes_{\min} \cA=C^*(F_{\infty}) \otimes_{\max} \cA$.  By \cite[Proposition 1.1(iii)]{kirchberg93}, $\cA$ has the WEP.
\end{proof}

There is a similar characterization for operator systems with the DCEP.

\begin{mythe}
\label{dcepv2}
Let $\cS$ be an operator system.  The following are equivalent.
\begin{enumerate}
\item
$\cS$ has the DCEP.
\item
$\cS \otimes_{\min} \cV_n=\cS \otimes_c \cV_n$ for all $n \geq 2$.
\item
$\cS \otimes_{\min} \cV_2=\cS \otimes_c \cV_2$.
\end{enumerate}
\end{mythe}

\begin{proof}
Assume that $\cS$ has the DCEP.  By Theorem \ref{dcep}, $\cS$ is $(\el,c)$-nuclear, while $\cV_n$ is $(\min,\er)$-nuclear.  It follows that $\cS \otimes_{\min} \cV_n=\cS \otimes_c \cV_n$ for all $n \geq 2$.  Hence, (1) implies (2).  Clearly (2) implies (3).  If (3) is true, then by Lemma \ref{vnmincunc} and by Theorem \ref{propertyvandw}, we must have $\cS \otimes_{\min} C^*(F_2)=\cS \otimes_{\max} C^*(F_2)$.  Since $C^*(F_{\infty})$ is a retract of $C^*(F_2)$, using Lemma \ref{retract} gives $\cS \otimes_{\min} C^*(F_{\infty})=\cS \otimes_{\max} C^*(F_{\infty})$.  Applying Theorem \ref{dcep} shows that $\cS$ has the DCEP, so that (1) is true.
\end{proof}

\section{Relating $\cV_n$ to Kirchberg's conjecture}

The proof of Theorem \ref{propertyvandw} shows that $C^*(F_n)$ is a retract of $\cU_{nc}(n)$ via ucp maps.  Using this trick allows for a connection between $\cU_{nc}(n)$ and Kirchberg's conjecture.

\begin{mythe}
\label{conditionforkirchberg}
If $\cU_{nc}(n) \otimes_{\min} \cU_{nc}(n) =\cU_{nc}(n) \otimes_{\max} \cU_{nc}(n)$ for some $n \geq 2$, then Kirchberg's conjecture is valid.
\end{mythe}

\begin{proof}
It is well known that Kirchberg's conjecture is true if and only if it holds for some $n \in \bN$ with $n \geq 2$.  Now, if $\cU_{nc}(n) \otimes_{\min} \cU_{nc}(n)=\cU_{nc}(n) \otimes_{\max} \cU_{nc}(n)$, then combining Lemmas \ref{retract} and \ref{uncretract} yields the complete order isomorphism $C^*(F_n) \otimes_{\min} C^*(F_n)=C^*(F_n) \otimes_{\max} C^*(F_n)$.
\end{proof}

The link between Kirchberg's conjecture and the WEP allows us to prove the converse of Theorem \ref{conditionforkirchberg}.  In other words, while the assumption that $\cU_{nc}(n) \otimes_{\min} \cU_{nc}(n)=\cU_{nc}(n) \otimes_{\max} \cU_{nc}(n)$ for some $n \geq 2$ appears to be slightly stronger than Kirchberg's conjecture, it is in fact equivalent to Kirchberg's conjecture.

\begin{mythe}
\label{uncequivalenttokirchberg}
The following statements are equivalent.
\begin{enumerate}
\item
$\cV_2 \otimes_{\min} \cV_2=\cV_2 \otimes_c \cV_2$.
\item
$\cU_{nc}(2) \otimes_{\min} \cU_{nc}(2)=\cU_{nc}(2) \otimes_{\max} \cU_{nc}(2)$.
\item
$C^*(F_2) \otimes_{\min} C^*(F_2)=C^*(F_2) \otimes_{\max} C^*(F_2)$.
\item
$C^*(F_{\infty}) \otimes_{\min} C^*(F_{\infty})=C^*(F_{\infty}) \otimes_{\max} C^*(F_{\infty})$.
\item
Connes' embedding problem has a positive answer.
\end{enumerate}
\end{mythe}

\begin{proof}
Note that if (3) holds, then since $C^*(F_{\infty})$ is a retract of $C^*(F_2)$, (4) also holds by the same argument as in the proof of Theorem \ref{weppropertyv2}.  Clearly $F_2$ embeds into $F_{\infty}$ so that, by \cite[Proposition 8.8]{pisierbook}, $C^*(F_2)$ is a retract of $C^*(F_{\infty})$.  Hence, (4) implies (3).  Using Lemma \ref{vnmincunc} shows that (1) implies (2), while Theorem \ref{conditionforkirchberg} shows that (2) implies (3).  Assuming (4) is true, it follows that $C^*(F_{\infty})$ has the WEP \cite{kirchberg93}.  Then \cite[Theorem~9.1]{quotients} shows that any operator system $\cS$ that is $(\min,\er)$-nuclear satisfies $\cS \otimes_{\min} \cS=\cS \otimes_c \cS$.  By Lemma \ref{bofhpropertyv} and Theorem \ref{osllp}, $\cV_2$ is $(\min,\er)$-nuclear.  Therefore, $\cV_2 \otimes_{\min} \cV_2=\cV_2 \otimes_c \cV_2$, as required.
\end{proof}

Because $\cS_n$ is a retract of $\cV_n$, we can prove the following.

\begin{pro}
\label{vncnotmax}
For all $n,m \geq 2$, $\cV_n \otimes_c \cV_m \neq \cV_n \otimes_{\max} \cV_m$.
\end{pro}

\begin{proof}
By \cite[Theorem 3.8]{FKPT}, we have $\cS_n \otimes_c \cS_m \neq \cS_n \otimes_{\max} \cS_m$ for all $n,m \geq 2$.  Hence, if $\cV_n \otimes_c \cV_m=\cV_n \otimes_{\max} \cV_m$, then by Lemmas \ref{retract} and \ref{uncretract}, we have $\cS_n \otimes_c \cS_m=\cS_n \otimes_{\max} \cS_m$, which is a contradiction.
\end{proof}

\begin{cor}
For all $n,m \geq 2$, $C_e^*(\cV_n \otimes_{\max} \cV_m) \neq C_e^*(\cV_n) \otimes_{\max} C_e^*(\cV_m)$.
\end{cor}

\begin{proof}
Suppose that $C_e^*(\cV_n \otimes_{\max} \cV_m)=C_e^*(\cV_n) \otimes_{\max} C_e^*(\cV_m)$.  The latter $C^{\ast}$-algebra is $\cU_{nc}(n) \otimes_{\max} \cU_{nc}(m)$.  Applying Proposition \ref{vncommutingorderembedding}, $\cV_n \otimes_c \cV_m$ is completely order isomorphic to its inclusion in $\cU_{nc}(n) \otimes_{\max} \cU_{nc}(m)$.  The operator system $\cV_n \otimes_{\max} \cV_m$ is completely order isomorphic to its inclusion in $C_e^*(\cV_n \otimes_{\max} \cV_m)$ \cite{hamana}.  Thus, $\cV_n \otimes_c \cV_m=\cV_n \otimes_{\max} \cV_m$, contradicting Proposition \ref{vncnotmax}.
\end{proof}

\begin{cor}
Let $U=(u_{ij}) \in M_n(\cU_{nc}(n))$ and $V=(v_{k\ell}) \in M_m(\cU_{nc}(m))$ be the matrices of generators of $\cU_{nc}(n)$ and $\cU_{nc}(m)$, respectively, where $n,m \geq 2$.  Then $U_0=(u_{ij} \otimes 1) \in M_n(C_e^*(\cV_n \otimes_{\max} \cV_m))$ and $V_0=(1 \otimes v_{k\ell}) \in M_m(C_e^*(\cV_n \otimes_{\max} \cV_m))$ fail to be unitary.
\end{cor}

\begin{proof}
Suppose that $U_0$ and $V_0$ were unitary.  The entries of $U_0$ $*$-commute with the entries of $V_0$, so there are unital $*$-homomorphisms $\pi:\cU_{nc}(n) \to C_e^*(\cV_n \otimes_{\max} \cV_m)$ and $\rho:\cU_{nc}(m) \to C_e^*(\cV_n \otimes_{\max} \cV_m)$ with $\pi(u_{ij})=u_{ij} \otimes 1$ and $\rho(v_{k\ell})=1 \otimes v_{k\ell}$.  The ranges of $\pi$ and $\rho$ commute, so that $\pi \cdot \rho:\cU_{nc}(n) \otimes_{\max} \cU_{nc}(m) \to C_e^*(\cV_n \otimes_{\max} \cV_m)$ is a $*$-homomorphism.  In particular, $\pi \cdot \rho$ is ucp and $\pi \cdot \rho$ is the identity map when restricted to $\cV_n \otimes \cV_m$.  By Proposition \ref{vncommutingorderembedding}, the inclusion of $\cV_n \otimes \cV_m$ into $\cU_{nc}(n) \otimes_{\max} \cU_{nc}(m)$ is $\cV_n \otimes_c \cV_m$.  Therefore, $\id:\cV_n \otimes_c \cV_m \to \cV_n \otimes_{\max} \cV_m$ is ucp, contradicting Proposition \ref{vncnotmax}.
\end{proof}

\section{Unitary Correlation Sets}

There is a way of stating Connes' embedding Problem in terms of what are known as quantum correlation matrices.  We outline the definition of these sets, as they will motivate the definition of a new collection of correlation sets.  The background for Tsirelson's problem in quantum correlations is motivated by a question in bipartite quantum information theory.  Essentially, there are two models often used for nonlocal quantum correlations: one is a tensor product model, and one is a commuting model.  Tsirelson's problem asks whether these models are the same, up to approximation.  We will formulate the sets of nonlocal quantum correlations in each model below.

If $\cH$ is a Hilbert space, we say that a set of operators $(P_i)_{i=1}^m$ is a \textbf{positive operator-valued measure with $m$ outputs} (POVM) if each $P_i \geq 0$ and $\sum_{i=1}^m P_i=I$.  If the $P_i$'s are also orthogonal projections, then we say that $(P_i)_{i=1}^m$ is a \textbf{projection-valued measure with $m$ outputs} (PVM).  Note that if $(P_i)_{i=1}^m$ is a PVM on $\cH$, then it necessarily follows that $P_i \perp P_j$ for $i \neq j$. For each choice of $n,m \in \bN$, the set of quantum commuting correlation probabilities of two separated systems of $n$ POVM's with $m$ outputs is given by $$C_{qc}(n,m)=\{ (\la P_{a,x} Q_{b,y} \xi,\xi \ra)_{a,b,x,y} \},$$
where for each $a,b \in \{1,...,n\}$, $(P_{a,x})_{x=1}^m$ and $(Q_{b,y})_{y=1}^m$ are POVM's with $m$ outputs, $\cH$ is some Hilbert space, $\xi \in \cH$ is a unit vector, and $P_{a,x}Q_{b,y}=Q_{b,y}P_{a,x}$ for all choices of $a,b,x,y$.  Similarly, we define $$C_q(n,m)=\{ (\la (P_{a,x} \otimes Q_{b,y})\xi,\xi \ra)_{a,b,x,y} \},$$
where each $(P_{a,x})_{x=1}^m$ is a POVM with $m$ outputs on a Hilbert space $\cH_A$, each $(Q_{b,y})_{y=1}^m$ is a POVM with $m$-outputs on a Hilbert space $\cH_B$, $\xi \in \cH_A \otimes \cH_B$ is a unit vector, and $\dim(\cH_A),\dim(\cH_B)<\infty$.  We also define the possibly larger set $C_{qs}(n,m)$ to be the set of all correlations with the same form as for $C_q(n,m)$, except that we allow the Hilbert spaces to be infinite-dimensional.  For convenience, we denote by $C_{qa}(n,m)$ the closure of $C_q(n,m)$.  It is known that $$C_q(n,m) \subseteq  C_{qs}(n,m) \subseteq C_{qa}(n,m) \subseteq C_{qc}(n,m), \, \forall n,m,$$
and $C_{qc}(n,m)$ is closed.  Moreover, each of these sets is convex.  One form of Tsirelson's problem, then, is determining whether $C_{qa}(n,m)=C_{qc}(n,m)$ for all $n,m \geq 2$.  More information on these correlation sets can be found in \cite{tsirelson}, \cite{fritz} and \cite{junge}.  It is shown in \cite{fritz} and \cite{junge} that, in the definitions of $C_q(n,m)$ and $C_{qc}(n,m)$, one may take the POVM's to simply be PVM's.

There is a natural link between the sets $C_{qa}(n,m)$, $C_{qc}(n,m)$ and states on tensor products of the $C^{\ast}$-algebra $C^{\ast}( \ast_n \bZ_m)$, where $\ast_n \bZ_m$ denotes the free product of $n$ copies of the finite cyclic group $\bZ_m$ (see \cite{fritz, junge,ozawa}).  One key fact is that $C^*(\ast_n \bZ_m)$ is isomorphic to $\ast_n \ell_m^{\infty}$, the free product of $n$ copies of $\ell_m^{\infty}$ \cite{fritz,junge,ozawa}.  If $g_x$ denotes the generator of the $x$-th copy of $\bZ_m$ in $C^*(\ast_n \bZ_m)$ and $e_{a,x}$ denotes the generator of the $a$-th coordinate in the $x$-th copy of $\ell_m^{\infty}$, then the isomorphism is implemented via $$g_x \mapsto \sum_{a=1}^m \exp \left( \frac{2\pi a i}{m} \right) e_{a,x}.$$
It follows (see \cite{fritz,junge,ozawa}) that $C_{qa}(n,m)$ is the set of coordinates in $\bR^{n^2m^2}$ given by the images of states $s \in \cS(C^*(*_n \bZ_m) \otimes_{\min} C^*(*_n \bZ_m))$ on the generating set $\{e_{a,x} \otimes e_{b,y}: 1 \leq a,b \leq m, \, 1 \leq x,y \leq n\}$, while $C_{qc}(n,m)$ corresponds to the images of states on $C^*(*_n \bZ_m) \otimes_{\max} C^*(*_n \bZ_m)$.  

For our purposes, we may consider the special case of $m=2$, which involves the $C^{\ast}$-algebra $C^*( \ast_n \mathbb{Z}_2)$.  Following the notation in \cite{FKPT}, we let $h_i$ be the generator of the $i$-th copy of $\bZ_2$ inside of $C^*(\ast_n \mathbb{Z}_2)$.  Each $h_i$ is a self-adjoint unitary.  We let $NC(n)$ be the operator system generated by $\{h_1,...,h_n\}$ inside of $C^*(*_n \bZ_2)$.

\begin{pro}
\emph{(Farenick-Kavruk-Paulsen-Todorov, \cite{FKPT})}
\label{ncboxunivprop}
If $X_1,...,X_n \in \bofh$ are hermitian contractions, then there is a unique ucp map $\gamma:NC(n) \to \bofh$ given by $\gamma(h_i)=X_i$ for all $1 \leq i \leq n$.
\end{pro}

The isomorphism $C^*(\ast_n \bZ_2) \simeq \ast_n \ell^{\infty}_2$ is implemented by the mapping $$h_i \mapsto p_i-q_i,$$
where $p_i$ is the element $(1,0)$ in the $i$-th copy of $\ell_2^{\infty}$, and $q_i$ is the element $(0,1)$ in the $i$-th copy of $\ell_2^{\infty}$. By \cite[Lemma 6.2]{FKPT}, the operator system $NC(n) \otimes_c NC(n)$ is completely order isomorphic to its image inside of $C^*(*_n \bZ_2) \otimes_{\max} C^*(*_n \bZ_2)$. With this information in hand, we easily obtain the following:

\begin{pro}
\label{qaqcorderiso}
For $n \geq 2$, $C_{qa}(n,2)=C_{qc}(n,2)$ if and only if the identity map $\id:NC(n) \otimes_{\min} NC(n) \to NC(n) \otimes_c NC(n)$ is an order isomorphism.
\end{pro}

\begin{pro}
\label{ncboxfactorsthroughvn}
For any $n \geq 2$, $NC(n)$ is a retract of $\cV_n$.
\end{pro}

\begin{proof}
By \cite[Proposition 5.7]{FKPT}, there are ucp maps $\eta:NC(n) \to \cS_n$ and $\theta:\cS_n \to NC(n)$ such that $\theta \circ \eta=\id_{NC(n)}$. By lemma \ref{uncretract}, there are ucp maps $\psi:\cS_n \to \cV_n$ and $\pi:\cV_n \to \cS_n$ with $\pi \circ \psi=\id_{\cS_n}$.  Then $\psi \circ \eta:NC(n) \to \cV_n$ and $\theta \circ \pi:\cV_n \to NC(n)$ are ucp maps satisfying $(\theta \circ \pi) \circ (\psi \circ \eta)=\id_{NC(n)}$.  We conclude that $NC(n)$ is a retract of $\cV_n$.
\end{proof}

We wish to define correlation matrices with respect to $\cU_{nc}(n)$ that are similar in nature to Tsirelson's correlation sets.  We recall that a $C^{\ast}$-algebra $\cA$ is said to be \textbf{residually finite-dimensional} if there is a family $(\pi_i)_{i \in I}$ of $*$-homomorphisms $\pi_i:\cA \to \cB(\cH_i)$ with each $\dim(\cH_i)<\infty$ such that $\pi:=\bigoplus_{i \in I} \pi_i$ is faithful.  A key component in linking the usual quantum correlation matrices with Kirchberg's conjecture is the fact that $C^*(F_n)$ is RFD for every $n$.  Here, we show that $\cU_{nc}(n)$ also enjoys this property.

\begin{mythe}
\label{uncrfd}
For any $n \geq 2$, $\cU_{nc}(n)$ is RFD.
\end{mythe}

\begin{proof}
The proof mimics the proof that $C^*(F_n)$ is RFD (see \cite[Theorem 7]{choi}).  It is not hard to see that $\cU_{nc}(n)$ is a separable $C^{\ast}$-algebra, so we may assume that $\cU_{nc}(n) \subseteq \bofh$ is faithfully represented on a separable infinite-dimensional Hilbert space $\cH$.  Hence, there are operators $U_{ij} \in \bofh$ for $1 \leq i,j \leq n$ such that $\cU_{nc}(n) \simeq C^*(\{U_{ij}\}_{i,j})$ via the mapping $u_{ij} \mapsto U_{ij}$.  Let $(P_m)_{m=1}^{\infty}$ be a sequence of increasing projections with $\text{rank}(P_m)=m$ and $SOT$-$\lim_{m \to \infty} P_m=I$.  Define $V_{m,ij}=P_mU_{ij}P_m$ and let $V_m=(V_{m,ij})$.  Since $\text{rank}(P_m)=m$, we may identify $V_{m,ij} \in M_m$ for each $i,j$ and hence $V_m \in M_n(M_m)$.  Observe that $$V_m=\begin{pmatrix} P_m \\ & \ddots \\ & & P_m \end{pmatrix} U \begin{pmatrix} P_m \\ & \ddots \\ & & P_m \end{pmatrix},$$
where $U=(U_{ij})$.  Therefore, each $V_m$ is a contraction. By Proposition \ref{univprop}, there exist unital $*$-homomorphisms $\pi_m:\cU_{nc}(n) \to M_2(M_m)$ for each $m \in \bN$ such that $$X_{m,ij}:=\pi_m(u_{ij})=\begin{pmatrix} V_{m,ij} & (\sqrt{I-V_mV_m^*})_{ij} \\ (\sqrt{I-V_m^*V_m})_{ij} & -V_{m,ji}^* \end{pmatrix}$$ 
for all $i,j$.  Since $V_{m,ij}^*=P_mU_{ij}^* P_m$, $SOT$-$\lim_{m \to \infty} V_m=U$ and $SOT$-$\lim_{m \to \infty} V_m^*=U^*$.  Hence, every entry of $V_m$ converges in SOT, so that $$SOT\text{-}\lim_{m \to \infty} X_m=\begin{pmatrix} U_{ij} & 0 \\ 0 & -U_{ji}^* \end{pmatrix}.$$
Let $F$ be any word in the generators of $\cU_{nc}(n)$.  We similarly obtain $$SOT\text{-}\lim_{m \to \infty} \pi_m(F)=\begin{pmatrix} F & 0 \\ 0 & F(\{-U_{ji}^*,-U_{ji}\}) \end{pmatrix},$$
where $F(\{-U_{ji}^*,-U_{ji}\})$ is the word obtained by replacing every occurrence of $U_{ij}$ with $-U_{ji}^*$, and every occurrence of $U_{ij}^*$ with $-U_{ji}$.  Assume that $F$ is norm $1$.  Then given $\ee>0$, there is $m_0 \in \bN$ such that for all $m \geq m_0$, we have $\|F(\{X_{m,ij},X_{m,ij}^*\})\| \geq 1-\ee$.  Hence, $\pi:=\bigoplus_{m \in \bN} \pi_m$ is isometric on the dense subspace of linear combinations of words in the generators of $\cU_{nc}(n)$.  Since $\pi$ must be continuous, $\pi$ is isometric on $\cU_{nc}(n)$.  This shows that $\pi$ is faithful and $\cU_{nc}(n)$ is RFD.
\end{proof}

\begin{rem}
\label{minrfd}
It is not hard to see that whenever $\cA$ and $\cB$ are RFD $C^{\ast}$-algebras, then $\cA \otimes_{\min} \cB$ is also RFD.  Hence, $\cU_{nc}(n) \otimes_{\min} \cU_{nc}(k)$ is RFD for every $n,k \geq 2$.
\end{rem}

As with $C^*(F_n)$, we can reformulate Kirchberg's conjecture in terms of whether or not $\cU_{nc}(n) \otimes_{\max} \cU_{nc}(n)$ is RFD.  The proof is identical to the $C^*(F_n)$ case \cite[Proposition 7.4.4]{brownozawa} and is omitted.

\begin{mythe}
\label{uncmaxrfdconjecture}
The following statements are equivalent.
\begin{enumerate}
\item
(Kirchberg's Conjecture) $C^*(F_n) \otimes_{\min} C^*(F_n)=C^*(F_n) \otimes_{\max} C^*(F_n)$ for all/some $n \geq 2$.
\item
$\cU_{nc}(n) \otimes_{\min} \cU_{nc}(n)=\cU_{nc}(n) \otimes_{\max} \cU_{nc}(n)$ for all/some $n \geq 2$.
\item
$\cU_{nc}(n) \otimes_{\max} \cU_{nc}(n)$ is RFD for all/some $n \geq 2$.
\end{enumerate}
\end{mythe}

We will show below that (3) holds if we weaken the assumption of residual finite-dimensionality to being quasidiagonal.  Recall that a $C^{\ast}$-algebra $\cA$ is \textbf{quasidiagonal}, or \textbf{QD}, if there is a net of ucp maps $\varphi_{\lambda}:\cA \to M_{k(\lambda)}$ such that $\lim_{\lambda} \|\varphi_{\lambda}(a)\|=\|a\|$ and $\lim_{\lambda} \|\varphi_{\lambda}(ab)-\varphi_{\lambda}(a)\varphi_{\lambda}(b)\|=0$ for all $a,b \in \cA$.  It is easy to see that every RFD $C^{\ast}$-algebra is QD.

\begin{mythe}
\label{uncmaxqd}
For every $n \geq 2$, $\cU_{nc}(n) \otimes_{\max} \cU_{nc}(n)$ is QD.
\end{mythe}

\begin{proof}
The proof is similar to the proof for $C^*(F_n) \otimes_{\max} C^*(F_n)$ (see \cite[Proposition 7.4.5]{brownozawa}).  Let $\pi:\cU_{nc}(n) \otimes_{\max} \cU_{nc}(n)$ be a faithful representation on a Hilbert space $\cH$.  Let $U=(U_{ij})$ be the matrix of generators of $\cU_{nc}(n) \otimes 1$, and let $V=(V_{ij})$ be the matrix of generators of $1 \otimes \cU_{nc}(n)$, so that each $U_{ij},V_{ij} \in \bofh$.  The nature of the max tensor product forces the $U_{ij}$'s and $V_{k\ell}$'s to $*$-commute.  The unitary group of $\cB(\cH^{(n)})$ is path connected by the Borel functional calculus.  Hence, there are norm-continuous functions $u,v:[0,1] \to \cB(\cH^{(n)})$ such that $u(0)=\cI_{\cH^{(n)}}=v(0)$, $u(1)=U$ and $v(1)=V$.  Since $UV=VU$ and $UV^*=V^*U$, the von Neumann algebras $W^*(U)$ and $W^*(V)$ must commute with each other.  Using the Borel functional calculus, we can arrange to have $u(t) \in W^*(U)$ and $v(t) \in W^*(V)$ for all $t \in [0,1]$.  The entries of $U$ and $V$ $*$-commute, so this must also hold for the entries of $p(U,U^*)$ and $q(V,V^*)$ for any $*$-polynomials $p,q$.  Taking limits, one sees that the entries of $u(t) \in \cB(\cH^{(n)})$ must $*$-commute with the entries of $v(t) \in \cB(\cH^{(n)})$ for all $t \in [0,1]$.  Since $u(t)$ and $v(t)$ are unitary with $*$-commuting entries, there is a unique $*$-homomorphism $\pi_t:\cU_{nc}(n) \otimes_{\max} \cU_{nc}(n) \to \bofh$ with $\pi_t(u_{ij} \otimes 1)=(u(t))_{ij}$ and $\pi_t(1 \otimes v_{ij})=(v(t))_{ij}$.  As $\pi_0$ is the trivial representation onto $\bC I_{\cH}$ and $\pi_1=\pi$, we see that $\pi$ is homotopic to the trivial representation.  Since $\pi$ is injective and $\bC$ is obviously QD, by \cite[Proposition 7.3.5]{brownozawa}, $\cU_{nc}(n) \otimes_{\max} \cU_{nc}(n)$ is QD.
\end{proof}

To obtain a unitary version of Tsirelson's problem that is equivalent to Kirchberg's conjecture, it is helpful to have a characterization of RFD $C^{\ast}$-algebras in terms of their state spaces.  By way of notation, we denote by $\cS(\cA)$ the set of all states on a unital $C^{\ast}$-algebra $\cA$.  We define $\text{Fin}(\cA)$ to be the set of all states on $\cA$ whose GNS representations act on finite-dimensional Hilbert spaces.  While a number of characterizations for residual finite-dimensionality are given in \cite{EL}, we only require the following one.

\begin{mythe}
\emph{(Exel-Loring \cite{EL})}
\label{EL}
A unital $C^{\ast}$-algebra $\cA$ is RFD if and only if $\text{Fin}(\cA)$ is $w^*$-dense in $\cS(\cA)$.
\end{mythe}

We are now in a position to define our unitary correlation sets.  As in the usual setting, we will consider a \textit{tensor product} model as well as a \textit{commuting} model.  For $n \geq 2$ and a unitary $U=(U_{ij}) \in M_n(\bofh)$ for some Hilbert space $\cH$, we let $\mathfrak{B}_n(U)=\{I_{\cH}\} \cup \{U_{ij},U_{ij}^*\}_{i,j=1}^n$.  We define $UC_q(n_1,n_2)$ to be the set of all $(2n_1^2+1)(2n_2^2+1)$-tuples of the form $$(\la X \otimes Y)\xi,\xi \ra)_{X \in \mathfrak{B}_{n_1}(U), \, Y \in \mathfrak{B}_{n_2}(V)},$$
where $U \in M_{n_1}(\cB(\cH_A))$ and $V \in M_{n_2}(\cB(\cH_B))$, $\cH_A$ and $\cH_B$ are finite-dimensional Hilbert spaces, and $\xi \in \cH_A \otimes \cH_B$ is a unit vector.  We define the possibly larger set $UC_{qs}(n_1,n_2)$ to be the set of all correlations of the same form as for $UC_q(n_1,n_2)$, except that we allow the Hilbert spaces to be infinite-dimensional.  For convenience, we will also define $UC_{qa}(n_1,n_2)$ to be the closure of $UC_q(n_1,n_2)$.  For the commuting unitary correlation sets, we define $UC_{qc}(n_1,n_2)$ to be the set of all $(2n_1^2+1)(2n_2^2+1)$-tuples of the form $$(\la XY\xi,\xi \ra)_{X \in \mathfrak{B}_{n_1}(U), \, Y \in \mathfrak{B}_{n_2}(V)},$$
where $U \in M_{n_1}(\bofh)$ and $V \in M_{n_2}(\bofh)$ are unitaries, $\cH$ is a Hilbert space, $\xi \in \cH$ is a unit vector, and $XY=YX$ for all $X \in \mathfrak{B}_{n_1}(U)$ and $Y \in \mathfrak{B}_{n_2}(V)$.  Since $U$ and $V$ are commuting unitaries, it follows that the $U_{ij}$'s and $V_{k\ell}$'s $*$-commute.  For convenience, we denote by $\mathcal{G}_{n_1,n_2}$ the set of generators of $\cV_{n_1} \otimes \cV_{n_2}$ of the form $x \otimes y$, where $x \in \{1\} \cup \{u_{ij},u_{ij}^*\}_{i,j=1}^{n_1}$ and $y \in \{1\} \cup \{v_{k\ell},v_{k\ell}^*\}_{k,\ell=1}^{n_2}$.  By the correspondence between GNS representations and states, $$UC_{qc}(n_1,n_2)=\{ (s(x))_{x \in \cG_{n_1,n_2}}: s \in \cS(\cU_{nc}(n_1) \otimes_{\max} \cU_{nc}(n_2)) \}.$$
By Proposition \ref{vncommutingorderembedding}, the inclusion $\cV_{n_1} \otimes_c \cV_{n_2} \subseteq \cU_{nc}(n_1) \otimes_{\max} \cU_{nc}(n_2)$ is a complete order embedding.  Therefore, we may also write $$UC_{qc}(n_1,n_2)=\{(s(x))_{x \in \cG_{n_1,n_2}}: s \in \cS(\cV_{n_1} \otimes_c \cV_{n_2})\}.$$
It is not hard to see that $$UC_q(n_1,n_2)=\{(s(x))_{x \in \cG_{n_1,n_2}}: s \in \text{Fin}(\cU_{nc}(n_1) \otimes_{\min} \cU_{nc}(n_2))\}.$$

These unitary correlation sets have similar properties to the quantum correlation sets.

\begin{pro}
\label{uccontainments}
For every $n_1,n_2 \geq 2$, $$UC_q(n_1,n_2) \subseteq UC_{qs}(n_1,n_2) \subseteq UC_{qa}(n_1,n_2) \subseteq UC_{qc}(n_1,n_2),$$ and each of these sets is convex.  Moreover, $UC_{qc}(n_1,n_2)$ is closed.
\end{pro}

\begin{proof}
Since the state space of any operator system is convex, it is easy to see that each set above is convex.  Clearly $UC_q(n_1,n_2) \subseteq UC_{qs}(n_1,n_2)$.  Every element of $UC_{qs}(n_1,n_2)$ corresponds to a state on $\cV_{n_1} \otimes_{\min} \cV_{n_2}$, which extends to a state on $\cU_{nc}(n_1) \otimes_{\min} \cU_{nc}(n_2)$ by the Hahn-Banach theorem.  By Theorem \ref{EL}, the set $\text{Fin}(\cU_{nc}(n_1) \otimes_{\min} \cU_{nc}(n_2))$ is $w^*$-dense in $\cS(\cU_{nc}(n_1) \otimes_{\min} \cU_{nc}(n_2))$, so that each element of $UC_{qs}(n_1,n_2)$ is also in $UC_{qa}(n_1,n_2)$.  To show that $UC_{qa}(n_1,n_2) \subseteq UC_{qc}(n_1,n_2)$, it suffices to show that $UC_{qc}(n_1,n_2)$ is closed.  To that end, let $((s_p(x)_{x \in \cG_{n_1,n_2}})_{p=1}^{\infty}$ be a sequence in $UC_{qc}(n_1,n_2)$ that converges, where $(s_p)_{p=1}^{\infty} \subseteq \cS(\cV_{n_1} \otimes_c \cV_{n_2})$.  The mapping $s:\cV_{n_1} \otimes_c \cV_{n_2} \to \bC$ given by $s(x)=\lim_{p \to \infty} s_p(x)$ for all $x \in \cG_{n_1,n_2}$ extends to a linear functional.  It follows that $s=w^*$-$\lim_{p \to \infty} s_p$.  Since the state space on an operator system is $w^*$-closed, we see that $s \in \cS(\cV_{n_1} \otimes_c \cV_{n_2})$ so that $UC_{qc}(n_1,n_2)$ is closed.
\end{proof}

Before we link these unitary correlation sets to Connes' embedding problem, it will be helpful to have a better description of $UC_{qa}(n_1,n_2)$.

\begin{lem}
\label{ucqa}
For each $n_1,n_2 \geq 2$, $$UC_{qa}(n_1,n_2)=\{(s(x))_{x \in \cG_{n_1,n_2}}: s \in \cS(\cV_{n_1} \otimes_{\min} \cV_{n_2})\}.$$
\end{lem}

\begin{proof}
Note that $\cV_{n_1} \otimes_{\min} \cV_{n_2}$ is completely order isomorphic to its inclusion in $\cU_{nc}(n_1) \otimes_{\min} \cU_{nc}(n_2)$.  Since $\cS(\cV_{n_1} \otimes_{\min} \cV_{n_2})$ is $w^*$-closed, the proof of Proposition \ref{uccontainments} shows that $$UC_{qa}(n_1,n_2) \subseteq \{(s(x))_{x \in \cG_{n_1,n_2}}: s \in \cS(\cV_{n_1} \otimes_{\min} \cV_{n_2})\}.$$
Conversely, let $s \in \cS(\cV_{n_1} \otimes_{\min} \cV_{n_2})$.  By the Hahn-Banach theorem we may extend $s$ to a state on $\cU_{nc}(n_1) \otimes_{\min} \cU_{nc}(n_2)$.  Since $\cU_{nc}(n_1) \otimes_{\min} \cU_{nc}(n_2)$ is RFD, by Theorem \ref{EL}, $s$ can be approximated pointwise by elements of $\text{Fin}(\cU_{nc}(n_1) \otimes_{\min} \cU_{nc}(n_2))$.  Restricting to the set $\cG_{n_1,n_2}$ yields a net of states whose images on the set $\cG_{n_1,n_2}$ are elements of $UC_q(n_1,n_2)$.  Since this net of states converges pointwise to $s$, we see that $(s(x))_{x \in \cG_{n_1,n_2}} \in UC_{qa}(n_1,n_2)$, as required.
\end{proof}

Using Lemma \ref{ucqa} allows us to formulate the problem of deciding whether $UC_{qa}(n_1,n_2)=UC_{qc}(n_1,n_2)$ in terms of $\cV_{n_1} \otimes_{\min} \cV_{n_2}$ and $\cV_{n_1} \otimes_c \cV_{n_2}$.

\begin{lem}
\label{utsirelsonorderiso}
Let $n_1,n_2 \geq 2$.  Then $UC_{qa}(n_1,n_2)=UC_{qc}(n_1,n_2)$ if and only if $\id:\cV_{n_1} \otimes_{\min} \cV_{n_2} \to \cV_{n_1} \otimes_c \cV_{n_2}$ is an order isomorphism.
\end{lem}

\begin{proof}
If $UC_{qa}(n_1,n_2)=UC_{qc}(n_1,n_2)$, then by linearity the states on $\cV_{n_1} \otimes_{\min} \cV_{n_2}$ and $\cV_{n_1} \otimes_c \cV_{n_2}$ are the same.  An element in an operator system is positive if and only if its image under each state is positive (see, for example, \cite[Chapter 13]{paulsen02}), so we conclude that $C_1^{\min}(\cV_{n_1},\cV_{n_2})=C_1^{\text{comm}}(\cV_{n_1},\cV_{n_2})$.  Therefore, $\cV_{n_1} \otimes_{\min} \cV_{n_2}$ and $\cV_{n_1} \otimes_c \cV_{n_2}$ must be order isomorphic.  Conversely, if $\id:\cV_{n_1} \otimes_{\min} \cV_{n_2} \to \cV_{n_1} \otimes_c \cV_{n_2}$ is an order isomorphism, then the positive elements are the same in the two operator systems, so the state spaces are identical.  Restricting to the set $\cG_{n_1,n_2}$, we obtain the equality $UC_{qa}(n_1,n_2)=UC_{qc}(n_1,n_2)$.
\end{proof}

We are now ready for the main result of this section.

\begin{mythe}
\label{unitarycorrelations}
The following are equivalent.
\begin{enumerate}
\item
Connes' embedding problem has a positive answer.
\item
$UC_{qa}(n_1,n_2)=UC_{qc}(n_1,n_2)$ for all $n_1,n_2 \geq 2$.
\item
$UC_{qa}(n,n)=UC_{qc}(n,n)$ for all $n \geq 2$.
\item
$C_{qa}(n,m)=C_{qc}(n,m)$ for all $n,m \geq 2$.
\item
$C_{qa}(n,2)=C_{qc}(n,2)$ for all $n \geq 2$.
\end{enumerate}
\end{mythe}

\begin{proof}
Suppose (1) holds.  By \cite[Theorem 9.1]{quotients}, Kirchberg's conjecture is equivalent to every $(\min,\er)$-nuclear operator system being $(\el,c)$-nuclear.  As each $\cV_n$ is $(\min,\er)$-nuclear, it follows that $\cV_{n_1} \otimes_{\min} \cV_{n_2}=\cV_{n_1} \otimes_c \cV_{n_2}$ for all $n_1,n_2 \geq 2$.  Hence, these operator systems are order isomorphic, so that $UC_{qa}(n_1,n_2)=UC_{qc}(n_1,n_2)$ for all $n_1,n_2 \geq 2$.
Clearly (2) implies (3) and (4) implies (5).  The implication $(5) \implies (1)$ was obtained by Ozawa \cite[Theorem 36]{ozawa}.  Hence, we need only show that (3) implies (5).  By Lemma \ref{utsirelsonorderiso}, condition (3) implies that $\cV_n \otimes_{\min} \cV_n$ and $\cV_n \otimes_c \cV_n$ are order isomorphic.  By Proposition \ref{ncboxfactorsthroughvn}, $NC(n)$ is a retract of $\cV_n$.  Using Lemma \ref{retract}, the identity map $\id:NC(n) \otimes_{\min} NC(n) \to NC(n) \otimes_c NC(n)$ is $1$-positive.  Since $\min \leq c$, we see that $NC(n) \otimes_{\min} NC(n)$ and $NC(n) \otimes_c NC(n)$ are order isomorphic for all $n \geq 2$.  Applying Proposition \ref{qaqcorderiso}, we obtain the equality $C_{qa}(n,2)=C_{qc}(n,2)$, as desired.
\end{proof}

Some striking differences arise between the quantum correlation sets and the unitary correlation sets.  It is known that $C_{qa}(2,2)=C_{qc}(2,2)$ (see, for example, \cite{ozawa}).  The question of whether $C_q(n,m)=C_{qc}(n,m)$ for all $n,m \geq 2$ was open until Slofstra \cite{slofstra} recently proved that there are large $n,m$ for which $C_{qs}(n,m) \neq C_{qc}(n,m)$. Similarly, it was unknown whether $C_{qs}(n,m)$ is closed for all $n,m \geq 2$, until Slofstra recently provided a counterexample \cite{slofstra17} for large $n,m$.

In contrast, it is now known that $UC_{qs}(2,2) \subsetneq UC_{qc}(2,2)$.  Indeed, in \cite{CLP} it is shown that there is a state $s:\cU_{nc}(2) \otimes_{\min} \cU_{nc}(2) \to \bC$ that cannot arise from a finite-dimensional representation of $\cU_{nc}(2) \otimes_{\min} \cU_{nc}(2)$.  In fact, it is shown that this state cannot arise from a spatial representation of $\cU_{nc}(2) \otimes_{\min} \cU_{nc}(2)$ on a tensor product of Hilbert spaces, even if the Hilbert spaces are infinite-dimensional.  Since $C_n^{\text{comm}}(\cU_{nc}(2),\cU_{nc}(2)) \subseteq C_n^{\min}(\cU_{nc}(2),\cU_{nc}(2))$, $s$ is also a state on $\cU_{nc}(2) \otimes_{\max} \cU_{nc}(2)$.  Hence we obtain an element of $UC_{qc}(2,2)$ that cannot be in $UC_{qs}(2,2)$.  This shows that $UC_{qs}(2,2) \subsetneq UC_{qc}(2,2)$.  Moreover, it is shown in \cite{CLP} that $s$ can be approximated in the $w^*$-topology by states on $\cU_{nc}(2) \otimes_{\min} \cU_{nc}(2)$ corresponding to elements of $UC_q(2,2)$.  Therefore, $UC_q(2,2)$ and $UC_{qs}(2,2)$ are not even closed.  The methods in \cite{CLP} can be adapted in a natural way to show that $UC_{qs}(n,n) \subsetneq UC_{qc}(n,n)$ for all $n \geq 2$, and that $UC_q(n,n)$ and $UC_{qs}(n,n)$ are not closed.

\end{document}